\documentclass[11pt,leqno]{amsart}
\usepackage{latexsym,amssymb,amsmath,multicol, rotating,lscape}

\usepackage{mathrsfs}

\usepackage{letltxmacro}
\usepackage{longtable}
\usepackage{tabularx}
\usepackage{ltablex}
\usepackage{multirow}
\usepackage{array}
\font\sm=msbm7 scaled \magstep 2
\font\germ=eufm10 scaled 1200

\newtheorem{thm}{Theorem}[section]

\newtheorem{prop}[thm]{Proposition}
\newtheorem{cor}[thm]{Corollary}
\newcommand{\thmref}[1]{Theorem~\ref{#1}}
\newcommand{\lemref}[1]{Lemma~\ref{#1}}
\newcommand{\rmkref}[1]{Remark~\ref{#1}}
\newcommand{\propref}[1]{Proposition~\ref{#1}}
\newcommand{\corref}[1]{Corollary~\ref{#1}}
\newcommand{\secref}[1]{Section~\ref{#1}}
\newcommand{\C}{{\mathbb C}}
\theoremstyle{remark}
\newtheorem{rmk}{Remark}[section]

\DeclareMathOperator{\lcm}{lcm}

\begin{document}

\title[Triangular numbers and forms of mixed type]
{On triangular numbers, forms of mixed type and their representation numbers}

\author{B. Ramakrishnan and Lalit Vaishya}
\address{Harish-Chandra Research Institute, HBNI,  
         Chhatnag Road, Jhunsi,
         Prayagraj (Allahabad) - 211 019,
         India.}
\email[B. Ramakrishnan]{ramki@hri.res.in}
\email[Lalit Vaishya]{lalitvaishya@hri.res.in}


\date{\today}
 
\maketitle

\begin{abstract}
 In \cite{ono}, K. Ono, S. Robins and P.T. Wahl considered the problem of determining formulas for the number of representations of a natural number $n$ by a sum of $k$ triangular numbers and derived many applications, including the one connecting these numbers with the number of representations of $n$ as a sum of $k$ odd square integers. They also obtained an application to the number of lattice points in the $k$-dimensional sphere. 

In this paper, we consider triangular numbers with positive integer coefficients. First we show that if the sum of these coefficients is a multiple of $8$, then the associated generating function gives rise to a modular form of integral weight (when even number of triangular numbers are taken). 
We then use the theory of modular forms to get the representation number formulas corresponding to the triangular numbers with coefficients. 
We also obtain several applications concerning the triangular numbers with coefficients similar to the ones obtained in \cite{ono}. 

In the second part of the paper, we consider more general mixed forms (as done in  Xia-Ma-Tian \cite{xia}) and derive modular properties for the corresponding generating functions associated to these mixed forms. Using our method we deduce all the 21 formulas proved in  \cite[Theorem 1.1]{xia} and show that our method of deriving the 21 formulas together with the $(p,k)$ parametrization of the generating functions of the three mixed forms imply the $(p,k)$ parametrization of the Eisenstein series $E_4(\tau)$ 
and its duplications. It is to be noted that the $(p,k)$ parametrization of $E_4$ and its duplications were derived by a different method by K. S. Williams and his co-authors. 
In the final section, we provide sample formulas for these representation numbers in the case of 4 and 6 variable forms.
\end{abstract}

\section{Introduction and Statement of results}
Finding formulas for the number of representations of a positive integer by certain class of quadratic forms is one of the 
classical problems in number theory, especially formulas for the number of representations of a natural number $n$ by a sum of 
$k$ integer squares, denoted by $r_k(n)$. In this connection, triangular numbers played an important role in finding formulas for 
$r_k(n)$, $8\vert k$ (see \cite{kohnen-imamoglu}). The $n$-th triangular number for a non-negative integer $n$ is given by $T_{n} := n(n+1)/2$. The first few triangular numbers are: $0,1,3,6,10,15,\ldots$. In 1995,  K. Ono, S. Robins and P.T. Wahl \cite{ono} considered the problem of determining formulas for the number of representations of a natural number $n$ by $k$ triangular numbers, denoted as $\delta_k(n)$. When $2\vert k$, 
they used the theory of modular forms of integral weight to get formulas for $\delta_k(n)$. In their work, they have also given several applications of the function $\delta_k(n)$.  One of them is the formula for $\delta_{24}(n)$,  involving the Ramanujan Tau function 
$\tau(n)$.  For general $k$, they related $\delta_k(n)$ with the number of representations of $n$ as a sum of $k$ odd square integers. Further, they also obtained applications to  the number of lattice points inside  $k$-dimensional sphere. In recent years there are many works which consider the problem of finding formulas for the representation numbers for some (known) quadratic forms with certain coefficients. We give here a few references to this effect \cite{{ramanujan}, {aalw}, {aalw1}, {w}, {apw}, {xia}, {rss-octonary}}. In a similar vein, in this paper, we consider triangular numbers with coefficients and obtain a criterion (which depends only on the coefficients) in order that the associated generating function gives rise to a modular form. 
Like in the work of Ono et. al, we also derive some applications related to the triangular numbers with coefficients. 
For the purpose of explicit formulas, we compute explicit bases for the vector space of modular forms of weight $k  = 2,3$  (in these cases we choose the coefficients such that the level of the space of modular forms is a divisor of $24$) and give sample formulas for the number of representations of $n$ by $2k$ triangular numbers with coefficients $1,2,3,4,6$. 

In the second part of this paper, following the work of E. X. W. Xia, Y. H. Ma and L. X. Tian \cite{xia}, we consider mixed forms 
and find their corresponding representation numbers. More specifically, in \cite{xia}, they consider the following type of forms in $2m$ variables, which is given as follows.  
\begin{equation}\label{mixed}
{\mathcal M}({\bf x}, {\bf y}, {\bf z}) :=  \sum_{i=1}^u a_i (x_{2i-1}^2 + x_{2i-1}x_{2i}+ x_{2i}^2) + \sum_{i=1}^v b_i y_i^2  +
\sum_{i=1}^k c_i T_{z_i}
\end{equation}
where ${\bf x} := (x_{1},x_{2}, \ldots , x_{2u})\in {\mathbb Z}^{2u}$, ${\bf y} := (y_{1}, y_{2}, \ldots , y_{v}) \in {\mathbb Z}^v$, 
${\bf z} := (z_{1}, z_{2}, \ldots , z_{k})\in {\mathbb N}_0^k$, ${\mathbb N}_0 = {\mathbb N} \cup \{0\}$. Further,  $2m = 
2u + v+ k$, $u,v,k$ are non-negative integers  such that $v+ k$ is even and the coefficients $a_i\in \{1,2,4,8\}$, $b_i\in \{1,2,3,6\}$ and $c_i\in \{1,2,3,4,6\}$. The above is a mixed form consisting of quadratic forms of type $m^2+mn+n^2$,  squares and triangular numbers with coefficients. 
 Though they defined a more general type of mixed forms ${\mathcal M}({\bf x}, {\bf y}, {\bf z})$, with coefficients $a_i$, $b_i$ and $c_i$ as above, in their work, they actually deal with 21 mixed forms (and obtained the corresponding representation formulas) with $m=4$ (i.e., 8 variable forms).  
In the second part, we consider forms of type \eqref{mixed} (all possible mixed forms) and show that the corresponding generating function is a modular form of weight $m/2$ with some level and character (depending on the coefficients $a_{i}, b_i$ and $c_{i}$). This enables us to use the 
theory of modular forms to get the required formulas corresponding to the mixed forms, including the 21 formulas obtained in \cite{xia}.
We also remark that in \cite{lam}, a few cases of triangular forms with coefficients and mixed forms consisting of squares and triangular numbers with some coefficients have been studied and formulas were obtained using different methods. 

\smallskip 

Before stating our main results, we make the following observation. In \cite{xia}, the $(p,k)$- parametrizations of $E_4(\tau)$, $\eta(\tau)$ (and their 
duplications) and certain theta function identities are used to find the required formulas for the 21 mixed forms. We prove that our 
formulas for the 21 mixed forms (which uses the theory of modular forms) along with the $(p,k)$-parametrizations of the three generating functions 
of the three forms appearing in  ${\mathcal M}({\bf x}, {\bf y}, {\bf z})$ imply the $(p,k)$-parametrizations of $E_4(\tau)$ (and its duplications). 
The $(p,k)$ parametrizations of $E_4(d\tau)$, $d\vert 12$, were obtained by a different method in \cite{{aw}, {aaw}}. 
A detailed presentation the above discussion appears in \S 5.1.

\smallskip

\noindent 
We are now ready to state our main results. Let us first fix some notations. \\

Let $\Psi(\tau)$ denote the generating function for the triangular numbers $T_n$. i.e., 
\begin{equation}\label{gen}
\Psi(\tau) := \sum_{n= 0}^{\infty} q^{T_{n}} = 1 + q + q^2 + q^3 + q^6 + q^{10} + \cdots, 
\end{equation}
where $q = e^{2 \pi i\tau}$ and so the generating function for $\delta_{k}(n)$ (representation number for the sum of $k$-triangular numbers ) is given by 
\begin{equation}
\Psi^{k} (\tau) = \sum_{n \ge 0} \delta_{k}(n) q^{n}.
\end{equation}
In our first theorem we consider the sum of $k$ triangular numbers with coefficients $c_{1}, c_{2}, \ldots, c_{k} \in {\mathbb N}_0$, which is 
given by 
\begin{equation}\label{triang}
T_{\mathcal C}({\bf z}) := \sum_{i=1}^{k} c_{i} \frac{z_{i}(z_{i}+1)}{2},
\end{equation}
where ${\bf z}$ is as above and ${\mathcal C}$ denotes the set of $k$-tuples $(c_1, c_2, \ldots, c_k)$, $c_i\le c_{i+1}$. 
Let us denote by $\delta_k({\mathcal C};n)$ as the number of representations of an integer $n\ge 0$ by $T_{\mathcal C}({\bf z})$. 
i.e.,
\begin{equation}
\delta_k({\mathcal C};n) =  \# \{(z_{1}, z_{2}, \ldots, z_{k}) \in {\mathbb N}_0^k \mid n = T_{\mathcal C}({\bf z}) \}.
\end{equation}
The generating function for $T_{\mathcal C}({\bf z})$,  the sum of $k$ -triangular numbers with coefficients $c_i$  is denoted by $\Psi_{\mathcal C}(\tau)$ and is given by
\begin{equation}\label{gen-coeff}
\Psi_{\mathcal C}(\tau) = \prod_{i=1}^k \Psi(c_i \tau) = \sum_{n=0}^\infty \delta_k({\mathcal C};n) q^n.
\end{equation}
Note that when all $c_i=1$, then  $\Psi_{\mathcal C}(\tau) = \Psi^k(\tau)$  and $\delta_k({\mathcal C};n) = \delta_k(n)$.

\smallskip

Let ${\bf h}:= c_1 + c_2 + \cdots + c_k$. Our first result shows that when ${\bf h}$ is a multiple of $8$, then $q^{{\bf h}/8} 
\Psi_{\mathcal C}(\tau)$ is a modular form of weight $k/2$.

\begin{thm}\label{modular}
For an even integer $k\ge 2$, let ${\bf h}$ be defined as above and set $N = 2\cdot lcm(c_1, c_2, \dots, c_k)$, ${\bf c} = c_1\cdot c_2 \cdots c_k$. Then, $q^{{\bf h}/8} \Psi_{\mathcal C}(\tau)$ is a modular form of weight $k/2$ on $\Gamma_0(N)$ with character 
$\chi = \left(\frac{(-1)^{k/2} 4{\bf c}}{\cdot}\right)$ if, and only if,  ${\bf h} \equiv 0 \pmod{8}$.
\end{thm}

\smallskip

The mixed form given by \eqref{mixed} has three components which are linear combinations of quadratic forms of type 
$m^{2} + mn + n^{2}$, squares and triangular numbers. For the second part of the paper, we take two sub classes of the above mixed forms which are combinations of triangular numbers with squares and forms of the type $m^{2} + mn + n^{2}$. 
We give a notation for these combinations in the following. 
\begin{equation}\label{st-lt}
\begin{split}
{\mathcal M }_{l, t}({\bf x}, {\bf z}) & = \sum_{i=1}^u a_i (x_{2i-1}^2 + x_{2i-1}x_{2i}+ x_{2i}^2)  + \sum_{i=1}^k c_i T_{z_i}, 
\quad k {\rm ~is~even~},\\
{\mathcal M }_{s, t}({\bf y}, {\bf z}) & =  \sum_{i=1}^v b_i y_i^2  + \sum_{i=1}^k c_i T_{z_i}, \quad v+k {\rm ~is~even~}, 
\end{split}
\end{equation}
with ${\bf x}, {\bf y}, {\bf z}, a_i, b_i, c_i$ as in \eqref{mixed}. 

\smallskip

The generating function for the mixed form ${\mathcal M}_{l,t}({\bf x}, {\bf z})$ is given as follows. 
$$
\psi_{l,t}(\tau) = \prod_{i=1}^u {\mathcal F}(a_i \tau) \prod_{j=1}^{k} \Psi(c_{j}\tau),
$$
where  ${\mathcal F}(\tau) = \sum_{m,n\in {\mathbb Z}} q^{m^2+mn+n^2}$. As $k$ is even, by \thmref{modular}, we know that $q^{{\bf h}/8} \prod_{j=1}^{k} \Psi(c_{j}\tau)$ is a modular form when ${\bf h} = \sum_ic_i$ is a multiple of $8$. Further, ${\mathcal F}(\tau)$ is a modular form in $M_1(\Gamma_0(3), \big(\frac{\cdot}{3}\big))$, which follows from Theorem 4 of \cite{schoen}. Combining these two modular properties we have the following theorem, which gives the modular property of $\psi_{l,t}(\tau)$.

\begin{thm}\label{lt}
Let $\psi_{l,t}(\tau)$ be the generating function for the mixed form ${\mathcal M }_{l, t}({\bf x}, {\bf z})$. Then $q^{{\bf h}/8} \psi_{l,t}(\tau)$ is a 
modular form in $M_{u+k/2}(\Gamma_0(L), \chi')$, where $L = lcm[3~ lcm(a_1, \ldots, a_u), 2~ lcm(c_1, \ldots, c_k)]$ and 
$\chi' = \chi$, if $u$ is even and $\chi' = \big(\frac{\cdot}{3}\big) \chi$, if $u$ is odd. 
\end{thm}

Next, we consider the mixed form ${\mathcal M}_{s,t}({\bf y}, {\bf z})$ and denote its generating function by $\psi_{s,t}(\tau)$. Then we have 
 \begin{equation}\label{gen;st}
\psi_{s,t}(\tau) =  \prod_{i=1}^{v} \theta(b_{i}\tau) \prod_{j=1}^{k} \Psi(c_{j}\tau),
\end{equation}
where $\theta(\tau)$ is the classical theta function given by $\theta(\tau) = \sum_{n\in {\mathbb Z}} q^{n^2}$. 

\smallskip

In the next theorem, we shall obtain the modularity of the function $\psi_{s,t}(\tau)$. To derive this, we use the fact that both the generating functions that appear on the right hand-side of \eqref{gen;st} are expressed as eta-quotients. 

\begin{thm}\label{odd-triangular}
	Let $v$ and $k$ be positive integers such that $v+k$ is even and ${\bf h}$ be the sum of the coefficients $c_i$. 
	Set $M = lcm[4~ lcm(b_1, \ldots, b_v), 2~ lcm(c_1, \ldots, c_k)]$. 
		 Then the function $q^{{\bf h}/8} \psi_{s,t}(\tau)$ is a modular form in $M_{(v+k)/2}(\Gamma_0(M), \chi'')$, 
 if, and only if, ${\bf h }\equiv 0 \pmod {8}$,
where  
$$
\chi'' = \begin{cases}
\left(\frac{(-1)^{(v+k)/2} 4 \prod_{i=1}^{v} b_{i}\prod_{j=1}^k c_{j}} {\cdot}\right),& {\rm ~if~} v {\rm ~and~} k {\rm ~are~even}, \\   
\left(\frac{(-1)^{(v+k)/2} 8 \prod_{i=1}^{v} b_{i}\prod_{j=1}^k c_{j}} {\cdot}\right),& {\rm ~if~} v {\rm ~and~} k {\rm ~are~odd}. \\   
\end{cases}
$$  
\end{thm}

\smallskip

Note that the generating function $\Phi(\tau)$ of the mixed form ${\mathcal M}({\bf x}, {\bf y}, {\bf z})$ is expressed as 
\linebreak 
$\prod_{i=1}^u {\mathcal F}(a_i \tau) \psi_{s,t}(\tau)$. Therefore, the modular property of ${\mathcal F}$ together with \thmref{odd-triangular} give the following result. 

\begin{cor}\label{mixed-m} 
The generating function $\Phi(\tau)$ associated to the mixed form ${\mathcal M}({\bf x}, {\bf y}, {\bf z})$ is a modular form upto a 
rational power of $q$. More precisely, $q^{{\bf h}/8} \Phi(\tau)$ is a modular form in $M_{u+(v+k)/2}(\Gamma_0(N_1), \omega)$, where $N_1 = lcm[3 lcm(a_1, \ldots, a_u), M]$ and $\omega =  \chi''$, if $u$ is even and $\omega = \big(\frac{\cdot}{3}\big) \chi''$, 
if $u$ is odd. Here $M$ and $\chi''$ are as in \thmref{odd-triangular}.
\end{cor}

\smallskip

Let $r_{v}({\mathcal B};n)$ denote the number of representations of $n$ by the sum of squares 
$\sum_{i=1}^v b_i y_i^2$, where ${\mathcal B}$ denotes the set of $v$-tuples $(b_1, \ldots, b_v)$. 
When all $b_i=1$, then $r_v({\mathcal B};n)$ is nothing but $r_v(n)$, the number of representations of 
$n$ as a sum of $v$ squares. As mentioned before, $\delta_{k}({\mathcal C};n)$ denotes the number of representations of $n$ by 
the linear combination of triangular numbers $\sum_{i=1}^k c_i T_{z_i}$. To simplify the notation, we will be writing the $v$-tuples and $k$-tuples  in the explicit examples as follows: $(b_1^{\alpha_1}, \ldots, b_ v^{\alpha_v})$, $(c_1^{e_1}, \ldots, c_k^{e_k}), \alpha_i, e_i \ge 0$, $\alpha_1+\cdots + \alpha_v =v$ and $e_1 + \cdots + e_k =k$. If some of the exponents are zero, then the corresponding coefficient is taken to be zero. For a natural number $n$, ${\mathcal N}_{s,t}(b_1^{\alpha_1}, \ldots, b_ v^{\alpha_v}; c_1^{e_1}, \ldots, c_k^{e_k}; n)$
denotes the number of representations of $n$ by the mixed form ${\mathcal M}_{s,t}({\bf y}, {\bf z})$. When all the $b_i$'s and $c_i$'s are equal to $1$, then it is denoted by  ${\mathcal N}_{s,t}(n)$.  Similarly, when all $b_i$'s are $1$, then $r_v({\mathcal B};n) = r_v(n)$ and 
when all $c_i$'s are equal to $1$, then $\delta_k({\mathcal C};n) = \delta_k(n)$.

\smallskip

\begin{rmk}
Let $v+k$ be an even positive integer. Then equation \eqref{gen;st} gives an expression for the generating function $\psi_{s,t}(\tau)$  in terms of the generating functions for the sum of squares with coefficients and the sum of triangular numbers with coefficients. Therefore, using the above notations for their representation numbers, we have 
\begin{equation*}
\begin{split}
\sum_{n \ge 0} {\mathcal N}_{s,t}(b_1^{\alpha_1}, \ldots, b_ v^{\alpha_v}; c_1^{e_1}, \ldots, c_k^{e_k}; n) q^n =   \big(\sum_{n \ge 0} r_{v}({\mathcal B};n)q^n \big) 
\big(\sum_{n \ge 0} \delta_{k}({\mathcal C};n) q^n\big).
\end{split}
\end{equation*}
Now, comparing the coefficient of $q^n$ both the sides, we get for $n \ge 1$,
\begin{equation*}
\begin{split}
{\mathcal N}_{s,t}(b_1^{\alpha_1}, \ldots, b_ v^{\alpha_v}; c_1^{e_1}, \ldots, c_k^{e_k}; n) 
 = \delta_{k}({\mathcal C} ;n)  + r_{v}({\mathcal B};n) + \sum_{m=1}^{n-1} r_{v}({\mathcal B};m)\delta_{k}({\mathcal C};n-m).
\end{split}
\end{equation*}
In other words, for $n \ge 1$, we have the following identity: 
\begin{equation}\label{st-formula}
\begin{split}
\delta_{k}({\mathcal C} ;n)  + r_{v}({\mathcal B};n)  = {\mathcal N}_{s,t}(b_1^{\alpha_1}, \ldots, b_ v^{\alpha_v}; c_1^{e_1}, \ldots, c_k^{e_k}; n) 
 - \sum_{m=1}^{n-1} r_{v}({\mathcal B};m)\delta_{k}({\mathcal C};n-m).
\end{split}
\end{equation}
In the case when all $b_i$'s and $c_i$'s are equal to 1, the above identity becomes   
\begin{equation}\label{st-formula1}
\begin{split}
\delta_{k}(n)+ r_{v}(n) = {\mathcal N}_{s,t}(n)- \sum_{m=1}^{n-1} r_{v}(m)\delta_{k}(n-m).
\end{split}
\end{equation}
Next we use \thmref{odd-triangular} to get some more information about the above representation numbers. Assume that 
${\bf h} = e_1 c_1+ e_2 c_2+\cdots + e_k c_k$ is a positive integer which is a multiple of $8$ and write it as ${\bf h} = 8p$. With this assumption, by \thmref{odd-triangular}, 
the function $q^p \psi_{s,t}(\tau)$ is a modular form of weight $(v+k)/2$ with some level $M$ and character $\chi''$ as in the theorem. Therefore, by using an explicit basis for the 
space of modular forms, it is possible to get a formula for the numbers ${\mathcal N}_{s,t}(b_1^{\alpha_1}, \ldots, b_ v^{\alpha_v}; c_1^{e_1}, \ldots, c_k^{e_k}; n)$, when 
$n\ge p$. Using this formula in \eqref{st-formula}, we see that a formula for $r_{v}({\mathcal B}; n)$ can be obtained recursively by knowing a formula for 
$\delta_{k}({\mathcal C}; n)$ and vice-versa. This is interesting in the case when both $v$ and $k$ are odd (with the condition on ${\bf h}$). 

\end{rmk}

\smallskip 
\bigskip

When $v=k$, we get a relation between the number of representations of a natural number $n$ by the linear combination (with coefficients $c_i$) of $k$ triangular numbers and the number of representations of $n$ as a linear combination (with coefficients 
$c_i$) of $k$ odd integer squares. By taking $v=k$, we let
\begin{equation}\label{qkn}
q_{k}({\mathcal C};n) := \# \{(y_{1}, y_{2}, \ldots, y_{k}) \in {\mathbb Z}^k \vert   n = \sum_{i=1}^k c_i y_i^2, y_{i} \ge 0, ~~~ 2\not\vert y_{i}\}.  
\end{equation}
Then we have the following result for $\delta_{k}({\mathcal C};n)$, which is an analogue of the corresponding result obtained  in \cite{ono} with all $c_i$'s are equal to 1. 

\begin{prop}\label{odd-square}
For a positive integer $n$, we have 
\begin{equation}\label{11}
\delta_{k}({\mathcal C};n)  = q_{k}({\mathcal C}; 8n+{\bf h}),
\end{equation}
where ${\bf h} = c_1 + \cdots + c_k$. In other words, the number of representations of $n$ by $T_{\mathcal C}({\bf z})$ is the same as the number of representations 
of $8n + {\bf h}$ as a sum of $k$ odd integer squares with coefficients $c_i$, $1\le i\le k$. 
In particular, when all $c_{i} =1$, we have $ {\bf h} = k$. So, the above gives
\begin{equation}\label{delta-odd square}
\delta_{k}(n)  = q_{k}(8n+k),
\end{equation}
which is \cite[Proposition 2]{ono}.
\end{prop}

We  now give the relation between $r_{k}(n)$, $\delta_{k}(n)$, $\delta_{2k}(n)$ by using the following identity (for a proof 
of this identity, we refer to \cite[p.40]{berndt}). 
\begin{equation}\label{psi-theta}
\psi^2(\tau) = \theta(\tau) \psi(2\tau).
\end{equation}
Taking $k$-th power of the above identity and using \eqref{delta-odd square}, we get the following corollary.

\begin{cor}\label{relations}
For a natural number $n$, we have 
\begin{equation}
q_{2k}(8n+2k) ~=~ \delta_{2k}(n) ~=~ 
\sum_{a,b\in {\mathbb N}_{0}\atop{a+ 2b =n}} r_{k}(a)\delta_{k}(b) ~= ~\sum_{a,b\in {\mathbb N}_{0} \atop{a+ 2b =n}} r_{k}(a) q_{k}(8b+k). 
\end{equation} 
\end{cor}

\begin{proof}
Taking $k$-th powers of both sides of \eqref{psi-theta} and comparing the respective $n$-th Fourier coefficients ($n\ge 1$), we get, 
\begin{equation*}
\delta_{2k}(n) = \sum_{a,b \in {\mathbb N}_{0}\atop{a+2b = n}} r_k(a) \delta_k(b) ~= \sum_{a,b \in {\mathbb N}_{0}\atop{a+2b = n}} 
r_k(a) q_k(8b+k) \quad ({\rm by~} \eqref{delta-odd square}).
\end{equation*}
Again by \eqref{delta-odd square}, we have $\delta_{2k}(n) = q_{2k}(8n+2k)$. 
\end{proof}

\begin{rmk}
If we consider \corref{relations} for forms with coefficients, then the identity \eqref{psi-theta} for the case of forms with coefficients  ${\mathcal C} = 
(c_1, \ldots, c_k)$ becomes: 
\begin{equation}\label{rmk:general}
\Psi_{\mathcal C}^2(\tau) = \prod_{i=1}^k\theta(c_i\tau) \Psi_{\mathcal C}(2\tau).
\end{equation}
The LHS can be viewed as the generating function of $2k$ triangular numbers with coefficients $(c_1,\ldots, c_k, c_1,\ldots, c_k) =: {\mathcal C}^2$ and the corresponding $n$-th Fourier coefficient is denoted as $\delta_{2k}({\mathcal C}^2;n)$. The number of representations of $n$ as a sum of $k$ squares with coefficients $(c_1,\ldots, c_k)$ is written in our notation as $r_k({\mathcal C};n)$. Therefore, comparing the $n$-th 
Fourier coefficients of \eqref{rmk:general}, we have the following identity: 
\begin{equation}\label{16}
\delta_{2k}({\mathcal C}^2;n) = \sum_{a,b\in {\mathbb N}_{0}\atop{a+ 2b =n}} r_{k}({\mathcal C}; a)\delta_{k}({\mathcal C}; b).
\end{equation}
When we consider $2k$ triangular numbers with coefficients $(c_1,\ldots,c_k,c_1,\ldots,c_k)$, the corresponding result for \propref{odd-square} 
is the following: 
\begin{equation}\label{17}
\delta_{2k}({\mathcal C}^2; n)  = q_{2k}({\mathcal C}^2; 8n+{2 \bf h}),
\end{equation}
where ${\bf h} = c_1+ \ldots + c_k$. Using \eqref{17} and \eqref{11} in \eqref{16}, the analogous result for \corref{relations} in the case of forms with coefficients is 
given by 
\begin{equation}
q_{2k}({\mathcal C}^2; 8n+{2 \bf h}) = \delta_{2k}({\mathcal C}^2; n) = \sum_{a,b\in {\mathbb N}_{0}\atop{a+ 2b =n}} r_{k}({\mathcal C}; a)\delta_{k}({\mathcal C}; b) = \sum_{a,b\in {\mathbb N}_{0}\atop{a+ 2b =n}} r_{k}({\mathcal C}; a) q_{k}({\mathcal C}; 8b + {\bf h}).
\end{equation}
\end{rmk}

\bigskip

We illustrate \corref{relations} with some examples. When $k=4,6,8,12$, and $16$, we use the formula for $\delta_k(n)$ given by Ono et. al in 
\cite[p. 77--81]{ono} to get the following relations among $\delta_k(n)$, $r_k(n)$. 

\begin{cor}\label{relations1}
 For a natural number $n$, we have
\begin{equation*}
\begin{split}
\delta_{4}(n)  = & \frac{1}{4} \sum_{a,b \in {\mathbb N}_{0}\atop{a+2b = n}} r_2(a) r_2(8b+2).\\
\delta_{6}(n)  = & \frac{1}{8} \sum_{a,b \in {\mathbb N}_{0}\atop{a+2b = n}} r_3(a) r_3(8b+3).\\
\delta_{8}(n)  = & \sum_{a,b \in {\mathbb N}_{0}\atop{a+2b = n}} r_4(a) \sigma(2b+1).\\
\delta_{12}(n) = & - \frac{1}{8} \sum_{a,b \in {\mathbb N}_{0}\atop{a+2b = n}} r_6(a) \sigma_{2; {\bf 1},\chi_{-4}}(4b+3).\\
\delta_{16}(n) = & \sum_{a,b \in {\mathbb N}_{0}\atop{a+2b = n}} r_8(a) \sigma_{3}^{\#}(b+1),\\
\end{split}
\end{equation*}
where $\sigma_{k;\chi,\psi}(n)$ is the generalised divisor function defined by \eqref{divisor}, $\chi_{-4}$ is the odd Dirichlet character modulo $4$  
and $\sigma_{3}^{\#}(n) = \displaystyle{\sum_{d \vert n \atop{\frac{n}{d}-odd}}} d^{3}$.
\end{cor}
 
\smallskip

\begin{proof}
The following formulas are established in \cite[p. 77--81]{ono}: 
\begin{equation*}
\begin{split}
\delta_{2}(n) ~=~ \frac{1}{4} r_2(8n+2); \quad  \delta_{3}(n) &= \frac{1}{8} r_3(8n+3); \quad \delta_{4}(n) ~=~ \sigma(2b+1),\\
\delta_{6}(n) ~=~ - \frac{1}{8} \sigma_{2;{\bf 1},\chi_{-4}}(4n+3);\quad  \delta_{8}(n) &= \sigma_{3}^{\#}(n+1). 
\end{split}
\end{equation*}
Using these formulas in \corref{relations}, we have the required formulas.
\end{proof}

\begin{rmk}
When $k=4,6,8$, we have the following known formulas for $r_k(n)$, which are given below. 
\begin{equation*}
\begin{split}
r_4(n) &= 8 \sigma(n)- 32 \sigma(n/4),\\
r_6(n) &= -4 \sigma_{2;{\bf 1}, \chi_{-4}}(n) + 16 \sigma_{2; \chi_{-4},{\bf 1}}(n),\\
r_8(n) &= 16 \sigma_{3}(n)- 32 \sigma_{3}(n/2) + 256 \sigma_{3}(n/4). 
\end{split}
\end{equation*}
Among these, $r_2(n)$ and $r_4(n)$ are well-known. To get the above expression for $r_6(n)$, we write $\theta^6(\tau)$ in terms of the 
two Eisenstein series $E_{3;{\bf 1}, \chi_{-4}}(\tau)$ and $E_{3;,\chi_{-4},{\bf 1}}(\tau)$ and it is easy to get the above formula by comparing the 
$n$-th Fourier coefficients. Now, using the above formulas in \corref{relations1}, we obtain the following formulas in terms of only the divisor functions.
For a natural number $n$, we have
\begin{equation*}
 \begin{split}
\delta_{8}(n)  &=  8 \sum_{a,b \in {\mathbb N}_{0}\atop{a+2b = n}} \sigma(a) \sigma(2b+1)-32 \sum_{a,b \in {\mathbb N}_{0}\atop{a+2b = n}} \sigma(a/4) \sigma(2b+1) .\\
\delta_{12}(n) &=   \frac{1}{2} \sum_{a,b \in {\mathbb N}_{0}\atop{a+2b = n}} \sigma_{2;{\bf 1},\chi_{-4}}(a) \sigma_{2;{\bf 1},\chi_{-4}}(4b+3) 
- 2 \sum_{a,b \in {\mathbb N}_{0}\atop{a+2b = n}} \sigma_{2;\chi_{-4},{\bf 1}}(a) \sigma_{2; {\bf 1},\chi_{-4}}(4b+3).\\
\delta_{16}(n) &=  16 \sum_{a,b \in {\mathbb N}_{0}\atop{a+2b = n}} \sigma_{3}(a) \sigma_{3}^{\#}(b+1) - 32 \sum_{a,b \in {\mathbb N}_{0}\atop{a+2b = n}} \sigma_{3}(a/2) \sigma_{3}^{\#}(b+1) + 256 \sum_{a,b \in {\mathbb N}_{0}\atop{a+2b = n}} \sigma_{3}(a/4) \sigma_{3}^{\#}(b+1).\\
\end{split}
\end{equation*}
\end{rmk}

\bigskip

\begin{prop}\label{ellipsoid}
Let $R$ be a positive real number. Then, the $k$-dimensional ellipsoid with axis lengths $R/\sqrt{c_{i}}$, $1\le i\le k$, 
centred at $(1/2,1/2, \ldots, 1/2)$ contains $2^{k} \sum_{n = 1}^{[\frac{R^2}{2}-\frac{{\bf h}}{8}]} 
\delta_{k}({\mathcal C};n)$ lattice points in ${\mathbb Z}^k$.
\end{prop}


\subsection{General formulas}

The main results in \S 1 give modular properties of the generating functions of the triangular numbers with coefficients and the mixed forms. In particular, from Theorems \ref{modular}, \ref{lt}, \ref{odd-triangular} and \corref{mixed-m} we know that the generating functions   
for $T_{\mathcal C}({\bf z})$, ${\mathcal M}_{l,t}({\bf x}, {\bf z})$, ${\mathcal M}_{s,t}({\bf y}, {\bf z})$ and ${\mathcal M}({\bf x}, {\bf y}, 
{\bf z})$ are all modular forms (upto a power of $q$) of integral weight for $\Gamma_0(N)$ with character $\chi$. The weight, level $N$ 
and character $\chi$ all depend on the coefficients chosen for the corresponding forms. Therefore, finding explicit basis for the space of modular forms of integral weight for some level and character is equivalent to getting explicit formulas for the corresponding representation 
numbers. We will elaborate this a little bit. 
Recall that the representation numbers corresponding to $T_{\mathcal C}({\bf z})$ and  ${\mathcal M}_{s,t}({\bf y}, {\bf z})$ are 
denoted respectively by $\delta_k({\mathcal C};n)$ and ${\mathcal N}_{s,t}(b_1^{\alpha_1}, \ldots, b_ v^{\alpha_v}; c_1^{e_1}, \ldots, c_k^{e_k}; n)$. We now give notations to the remaining two forms. Let ${\mathcal N}_{l,t}(a_1^{\beta_1}, \ldots, a_ u^{\beta_u}; c_1^{e_1}, \ldots, c_k^{e_k}; n)$ denote the representation number for the form ${\mathcal M}_{l,t}({\bf x}, {\bf z})$ and ${\mathcal N}(a_1^{\beta_1}, \ldots, a_ u^{\beta_u}; b_1^{\alpha_1}, \ldots, b_ v^{\alpha_v};c_1^{e_1}, \ldots, c_k^{e_k}; n)$ denote the representation 
number for the form ${\mathcal M}({\bf x}, {\bf y}, {\bf z})$, where $2u+v+k=2m$. Now, let ${\mathcal G}(\tau)$ be one of the generating functions ${\Psi}_{\mathcal C}(\tau)$, $\psi_{l,t}(\tau)$, $\psi_{s,t}(\tau)$ or ${\Phi}(\tau)$ (corresponding to the forms 
 $T_{\mathcal C}({\bf z})$, ${\mathcal M}_{l,t}({\bf x}, {\bf z})$, ${\mathcal M}_{s,t}({\bf y}, {\bf z})$ or ${\mathcal M}({\bf x}, {\bf y}, {\bf z})$ respectively). Then by Theorems \ref{modular}, \ref{lt}, \ref{odd-triangular} and \corref{mixed-m}, we see that $q^{{\bf h}/8} 
 {\mathcal G}(\tau)$ belongs to the space $M_m(\Gamma_0(N),\tilde{\chi})$, where $2u+v+k =2m$, 
 $$
 N = \begin{cases}
2\cdot lcm(c_1, c_2, \dots, c_k) & {\rm ~for~} {\Psi}_{\mathcal C}(\tau),\\
 lcm[3~ lcm(a_1, \ldots, a_u), 2~ lcm(c_1, \ldots, c_k)] & {\rm ~for~} \psi_{l,t}(\tau),\\
 lcm[4~ lcm(b_1, \ldots, b_v), 2~ lcm(c_1, \ldots, c_k)] & {\rm ~for~} \psi_{s,t}(\tau),\\
lcm[3 lcm(a_1, \ldots, a_u), lcm[4~ lcm(b_1, \ldots, b_v), 2~ lcm(c_1, \ldots, c_k)]]  & {\rm ~for~} \Phi(\tau).\\
\end{cases}
$$  
and 
$$
\tilde{\chi} = \begin{cases}
 \chi&  {\rm ~for~} {\Psi}_{\mathcal C}(\tau),\\
 \chi' &  {\rm ~for~} \psi_{l,t}(\tau),\\
\chi'' & {\rm ~for~} \psi_{s,t}(\tau),\\
\omega &{\rm ~for~} \Phi(\tau),\\
\end{cases}
$$
where $\chi$, $\chi'$, $\chi''$ and $\omega$ are as in Theorems \ref{modular}, \ref{lt}, \ref{odd-triangular} and \corref{mixed-m}, respectively and \\${\bf h} = (c_1 + \ldots +c_k)$. If $\{f_1, \ldots, f_{\nu_{m,N,\tilde{\chi}}}\}$ is a basis for the space $M_m(\Gamma_0(N), \tilde{\chi})$, then we can express $q^{{\bf h}/8} {\mathcal G}(\tau)$ in terms of this basis as follows. 
\begin{equation}\label{formula}
q^{{\bf h}/8} {\mathcal G}(\tau) = \sum_{i=1}^{\nu_{m,N,\tilde{\chi}}} t_i f_i(\tau).
\end{equation}
Let $p = {\bf h}/8$ and denote by $a_{f_i}(n)$, the $n$-th Fourier coefficient of the $i$-th basis element $f_i$. Then comparing the $n$-th 
Fourier coefficients in both the sides of \eqref{formula}, we get 
\begin{equation}\label{formula1}
 \sum_{i=1}^{\nu_{m,N,\tilde{\chi}}} t_i a_{f_i}(n) = \begin{cases}
\delta_k({\mathcal C};n-p) & {\rm ~for~} {\Psi}_{\mathcal C}(\tau),\\
{\mathcal N}_{l,t}(a_1^{\beta_1}, \ldots, a_ u^{\beta_u}; c_1^{e_1}, \ldots, c_k^{e_k}; n-p) &  {\rm ~for~} \psi_{l,t}(\tau),\\
{\mathcal N}_{s,t}(b_1^{\alpha_1}, \ldots, b_ v^{\alpha_v}; c_1^{e_1}, \ldots, c_k^{e_k}; n-p) &  {\rm ~for~} \psi_{s,t}(\tau),\\
 {\mathcal N}(a_1^{\beta_1}, \ldots, a_ u^{\beta_u}; b_1^{\alpha_1}, \ldots, b_ v^{\alpha_v};c_1^{e_1}, \ldots, c_k^{e_k}; n-p) &  {\rm ~for~} \Phi(\tau).\\
 \end{cases}
 \end{equation}



In \S 3, we use the above method to derive the 21 formulas obtained by Xia et al in \cite{xia}.

\smallskip

\subsubsection{ Sample and Explicit formulas}

In the above section we described the method to get explicit formulas for the representation numbers. 
Our method is illustrated with some specific formulas for the case where the level $N$ is a divisor of $24$ and the 
weight of the modular forms space is either $2$ or $3$. In \S 4, we give a list of forms for the purpose of explicit examples. Forms with 4 variables are given 
in Table 2(a) and forms with 6 variables are given in Table 2(b). In this section, we also provide explicit basis for the spaces $M_2(\Gamma_0(d),\chi)$ (Table A) and 
$M_3(\Gamma_0(d),\chi)$ (Table B), where $d\vert 24$. 
For the remaining cases listed in Tables 2(a) and 2(b), one can derive the explicit formulas using \eqref{formula1} and the 
coefficients ($t_i$'s) appearing  in Tables 3 to 14 (Appendix). 
We have used the open-source mathematics software SAGE (www.sagemath.org) for carrying out the above  calculations.

\bigskip

\bigskip

\section{Proofs of theorems}

Before we proceed to give proofs of our theorems, we shall present some known facts. 
For a positive integers $k$ and $M$, let $M_k(\Gamma_0(M),\psi)$ denote the ${\mathbb C}$-vector space of holomorphic modular forms of weight $k$ with respect to the congruence subgroup $\Gamma_0(M)$ with character $\psi$ (which is a Dirichlet character modulo $M$). The subspace of cusp forms is denoted by $S_k(\Gamma_0(M), \psi)$. Let $\eta(\tau) = q^{1/24}\prod_{n\ge 1} (1-q^n)$ be the Dedekind eta function, where $q = e^{2\pi i\tau}$, $\tau\in {\mathcal H}$, the complex upper half-plane. 
An eta-quotient is defined as $\prod_{\delta\vert M} \eta^{r_\delta}(\delta \tau)$, where the product varies over all the positive divisors $\delta$ of $M$ and $r_\delta$ are integers. The reason for naming this as an eta-quotient is the fact that some of the eta powers $r_\delta$ can be negative also. The following result gives necessary conditions for an eta-quotient to be a modular form of integral weight.\\

\smallskip

\noindent {\bf Theorem A} (Dummit, Kisilevsky and McKay \cite{dummit}) \\
{\em 
For $M \in {\mathbb N}$, let $r_\delta \in {\mathbb Z}$ for $\delta \vert M $. Let $f(\tau) = \displaystyle{\prod_{\delta \vert M}\eta^{r_\delta}(\delta \tau)}$ be an eta-quotient such that 
$k = \frac{1}{2}\displaystyle{\sum_{\delta \vert M} r_\delta}$ is a positive integer. If $f(\tau)$ satisfies the the following conditions 
\begin{enumerate}
\item[{(i)}]
$\displaystyle{\sum_{\delta \vert M}\delta r_\delta \equiv 0\pmod{24}}$ ~and ~
$\displaystyle{\sum_{\delta \vert M}\frac{M}{\delta} r_\delta \equiv 0 \pmod{24}}$,
\item[{(ii)}]
For each $ d\ge 1$, $\displaystyle{\sum_{\delta \vert M}\frac{\\gcd{(d,\delta)}^2 r_\delta}{\delta}} \geq 0$,
\end{enumerate}
then $f(\tau)$ is a modular form in $M_k(\Gamma_0(M),\psi)$, where the Dirichlet character $\psi$ is given by 
$\left(\frac{{(-1)^k}s}{\cdot}\right)$, where $s =\displaystyle{\prod_{\delta \vert M}\delta^{|r_\delta|}}$.  In addition to above 
conditions, if in (ii), the quantity is strictly positive for each $d\vert M$,  then $f(\tau)\in S_k(\Gamma_0(M),\psi)$.\\
}

\smallskip

We also make use of the following fact, which expresses the generating function for the triangular numbers (given by \eqref{gen}) 
as an eta-quotient (for details we refer to \cite[Proposition 1]{ono}). 
\begin{equation}\label{gen-eta}
\Psi(\tau) = \prod_{n=1}^{\infty} \frac{(1-q^{2n})^2}{(1-q^{n})}~
= q^{-1/8} \frac{\eta^{2}(2\tau)}{\eta(\tau)}.
\end{equation}

\bigskip

\subsection{Proof of \thmref{modular}}
We use the expression for the generating function $\Psi(\tau)$ in terms of eta-quotient and \eqref{gen-coeff} to prove this theorem 
with the help of Theorem A. Using \eqref{gen-coeff} and \eqref{gen-eta}, we have 
\begin{equation}
\Psi_{\mathcal C}(\tau)  = \prod_{i=1}^k \Psi(c_i \tau) ~
= q^{-{\bf h}/8} \prod_{i=1}^k \frac{\eta^2(2c_i \tau)}{\eta(c_i \tau)}.
\end{equation}
Therefore, $q^{{\bf h}/8} \Psi_{\mathcal C}(\tau)$ is the eta-quotient $\displaystyle{\prod_{i=1}^k 
\frac{\eta^2(2c_i \tau)}{\eta(c_i \tau)}}$. Now we use Theorem A to get the required modular property. It is easy to check that for condition (i) of Theorem A to satisfy, we need the property that ${\bf h} = c_1+\cdots+c_k \equiv 0\pmod{8}$. To check condition (ii), 
we note that the pairs $(\delta, r_\delta)$ are given by $(c_i, -1)$ and $(2c_i, 2)$, $1\le i\le k$. Therefore, we have 
\begin{equation*}
\sum_{\delta\vert N} \frac{\\gcd(d,\delta)^2 r_\delta}{d \delta}  = \sum_{i=1}^k \left(2 \frac{\\gcd(d,2c_i)^2}{d (2c_i)} - \frac{\\gcd(d,c_i)^2}{d c_i} \right).
\end{equation*}
Since for every $d\ge 1$, we have $\\gcd(d,2c_i) \ge gcd(d, c_i)$, the right-hand side of the above equation is non-negative, condition (ii) of Theorem A is satisfied for this eta-product. The weight and character parts can be checked easily. 
This completes the proof.

\subsection{Proof of \thmref{odd-triangular}}
Since $\theta(\tau) = \eta^5(2\tau)/(\eta^2(\tau)\eta^2(4\tau))$, using \eqref{gen;st} and \eqref{gen-eta}, we have
\begin{equation*}
\begin{split}
\psi_{s,t}(\tau) & =  \prod_{i=1}^{v} \theta(b_{i}\tau).\prod_{j=1}^{k} \Psi(c_{j}\tau) \\ & =\prod_{i=1}^v \frac{\eta^5(2b_i \tau)}{\eta^2(b_i \tau)\eta^2(4b_i \tau)} . q^{-{\bf h}/8} \prod_{j=1}^k \frac{\eta^2(2c_j \tau)}{\eta(c_j \tau)}  .
\end{split}
\end{equation*}
Therefore, we have 
\begin{equation} \label{sq-triang}
q^{{\bf h}/8} \psi_{s,t}(\tau) = \displaystyle{\prod_{i=1}^v \frac{\eta^5(2b_i \tau)}{\eta^2(b_i \tau)\eta^2(4b_i \tau)} \prod_{i=1}^k \frac{\eta^2(2c_i \tau)}{\eta(c_i \tau)}}.
\end{equation}

As in the proof of \thmref{modular}, we get the modular property of this eta-quotient by using Theorem A. 
To get condition (i), we need the property $c_1+\cdots+c_k \equiv 0\pmod{8}$. To check condition (ii), we use the following fact: \\
\noindent {\bf Fact}: For positive integers $d$ and $b$, we have 
$$
10 \gcd(d,2b)^{2} - 8 \gcd(d,b)^{2} -2 \gcd(d,4b)^{2} \ge 0.
$$
Write $d= 2^\alpha d_1$ and $b= 2^\beta b_1$, with $\alpha, \beta \ge 0$ and $2\not\vert d_1b_1$. If $S$ denotes the LHS of the 
above property, then it is easy to see that $S=0$ when $\alpha\le \beta$ and $\alpha> \beta+1$ and in the remaining case $S>0$.
This proves the fact and this gives the required condition (ii) of Theorem A for our eta-quotient. 
The weight and character of the resulting modular form can be obtained easily by using Theorem A. 
This completes the proof.

\smallskip

\smallskip

\subsection{Proof of \propref{odd-square}}

Let $n\ge 1$ be an integer represented by $T_{\mathcal C}({\bf z})$. Then we have 
\begin{equation*}
n  = \sum_{i=1}^{k} c_{i} \frac{z_{i}(z_{i}+1)}{2},
\end{equation*} 
and so $8n = 4\sum_{i=1}^k c_i z_i(z_i+1)$. Since ${\bf h} = c_1+\ldots + c_k$, we get  
\begin{eqnarray*}
8n + {\bf h} & =  \sum_{i=1}^{k} c_{i} (4x_i^2 + 4x_i+1) & = \sum_{i=1}^k c_i (2x_i+1)^2.
\end{eqnarray*}
Since one can trace back these identities, it follows that $\delta_{k}({\mathcal C};n)  = q_{k}({\mathcal C}; 8n+{\bf h})$.

\smallskip

\subsection{Proof of \propref{ellipsoid}}
Let $r$ be a positive integer and assume that the $k$-dimensional ellipsoid with axis lengths $r/\sqrt{c_{i}}$, $1\le i\le k$ centred at $(1/2, \ldots, 1/2)$ 
contains a lattice point $(z_{1}, z_{2}, \ldots, z_{k}) \in {\mathbb Z}^k$, then we have 
\begin{equation*}
r^2 ~ =~ \sum_{i=1}^{k} c_{i} (z_{i}-\frac{1}{2})^2 ~= \sum_{i=1}^{k} c_{i} [(z_{i}^2 - z_{i}) + 1/4] ~=  \sum_{i=1}^{k} c_{i} z_{i}(z_{i}-1) 
+ \frac{1}{4} \sum_{i=1}^{k} c_{i} ~=~\sum_{i=1}^{k} c_{i} z_{i}(z_{i}-1) + \frac{{\bf h}}{4}. \\
\end{equation*}
Therefore, $\frac{r^2}{2} -  \frac{{\bf h}}{8} = \sum_{i=1}^{k} c_{i} \frac{z_{i}(z_{i}-1)}{2}$. 
This shows that we have a representation of $\frac{r^2}{2}- \frac{{\bf h}}{8}$ as a sum of $k$ triangular numbers with coefficients $c_i$, $1\le i\le k$. Since $z_i$ and $(-z_i-1)$ represent the same triangular number, the total 
number of lattice points in this ellipsoid is equal to  $2^{k} \delta_{k}(\mathcal{C};[\frac{r^2}{2}-\frac{{\bf h}}{8}])$. 
Thus, we have the required formula mentioned in the proposition. 

\smallskip

\section{Formulas for the $21$ mixed forms and the $(p,k)$-parametrization of the Eisenstein series $E_4(d\tau)$, $d\vert 12$.} 

In this section, we first derive formulas for the 21 mixed forms considered in the work of Xia et. al \cite{xia} using our method. In \corref{mixed-m} we observed that the generating function $q^{{\bf h}/8} \Phi(\tau)$ corresponding to the mixed form ${\mathcal M}({\bf x}, {\bf y}, {\bf z})$ is a modular form of weight $u+(v+k)/2$ on some 
subgroup $\Gamma_0(N)$ with character $\omega$, depending on the coefficients involved in the mixed form. 
In the above, {\bf h} is the sum of the coefficients appearing in the triangular part of the mixed form.  For the $21$ cases, it turns out that all the generating 
functions lie in the space $M_4(\Gamma_0(12))$.  We use the following basis for the space $M_4(\Gamma_0(12))$ (obtained in \cite{rss-octonary}): 
$$
E_4(d\tau), d\vert 12; f_{4,6}(\tau), f_{4,6}(2\tau), f_{4,12}(\tau),
$$
where $E_4(\tau)$ is the normalized Eisenstein series of weight $4$ for $SL_2({\mathbb Z})$ and $f_{4,6}(\tau)$, $f_{4,12}(\tau)$ are 
the newforms of weight $4$ on $\Gamma_0(6)$ and $\Gamma_0(12)$ respectively, which are given explicitly in terms of the eta-functions as 
follows.  
\begin{equation*}
\begin{split}
f_{4,6}(\tau) & = \eta^2(\tau) \eta^2(2\tau) \eta^2(3\tau) \eta^2(6\tau), \\
f_{4,12}(\tau)& = \frac{\eta^2(2\tau) \eta^3(3\tau) \eta^3(4\tau)\eta^2(6\tau)}{\eta(\tau) \eta(12\tau)} - 
\frac{\eta^3(\tau) \eta^2(2\tau) \eta^2(6\tau)\eta^3(12\tau)}{\eta(3\tau) \eta(4\tau)}.
\end{split}
\end{equation*}
So, we express the generating functions $q^{{\bf h}/8} \Phi(\tau)$ as a linear combination of the above basis elements as 
\begin{equation}\label{21mixed}
q^{{\bf h}/8} \Phi(\tau) = \sum_{d\vert 12} \alpha_d E_4(d\tau) + c_1 f_{4,6}(\tau) + c_2 f_{4,6}(2\tau) + c_3 f_{4,12}(\tau).
\end{equation}

In the following table, we give the mixed forms and the corresponding coefficients $\alpha_d$, $d\vert 12$, $c_i$, $i=1,2,3$.

\smallskip

\newpage
\begin{center}

{\bf Table 1. Coefficients table for the 21 mixed forms}

\smallskip

\renewcommand{\arraystretch}{1.2}
\begin{tabularx}{\textwidth}{|p{5.25cm}| p{0.8cm}|p{0.8cm}|p{0.8cm}|p{0.8cm}|p{0.8cm}|p{0.8cm}|p{0.8cm}|p{0.8cm}|p{0.8cm}|}

\hline
     
                            \multicolumn{1}{|c|}{\textbf{Form}}  & \multicolumn{9}{|c|}{\textbf{Coefficients}}\\ \cline{2-10} 
{} & {$\alpha_1$} & {$\alpha_2$} & {$\alpha_3$} & {$\alpha_4$} & {$\alpha_6$} & {$\alpha_{12}$} & {$c_1$} & {$c_2$} & {$c_3$}  
     \\ \hline 
     
\endfirsthead
     \caption{Coefficients for the 21 mixed forms}\\
      \hline
                              \multicolumn{1}{|c|}{\textbf{Form}}  & \multicolumn{9}{|c|}{\textbf{Coefficients}}\\ \cline{2-10} 
    {} & {$\alpha_1$} & {$\alpha_2$} & {$\alpha_3$} & {$\alpha_4$} & {$\alpha_6$} & {$\alpha_{12}$} & {$c_1$} & {$c_2$} & {$c_3$}  
     \\ \hline 
     
\endhead
      \hline
      \multicolumn{10}{|r|}{{Continued on Next Page\ldots}}     \\ \hline
      
\endfoot

\endlastfoot
\hline

$ {\mathcal F}(2\tau)                           \theta^{3}(\tau) \theta^{3}(3\tau)                                              $ & $  \frac{1}{120} $  & $  0 $  & $  \frac{-3}{40} $  & $  \frac{-2}{15} $  & $  0 $  & $  \frac{6}{5} $  & $  0 $  & $  0 $  & $  4 $  \\ \hline 
$ q{\mathcal F}(2\tau)                           \theta^{2}(\tau) \theta^{2}(3\tau) \Psi(2\tau) \Psi(6\tau)                     $ & $  \frac{1}{480} $  & $  \frac{-1}{480} $  & $  \frac{-3}{160} $  & $  0 $  & $  \frac{3}{160} $  & $  0 $  & $  0 $  & $  0 $  & $  \frac{1}{2} $  \\ \hline 
$ q^{2}{\mathcal F}(2\tau)                           \theta(\tau) \theta(3\tau) \Psi^{2}(2\tau) \Psi^{2}(6\tau)                 $ & $  \frac{1}{1920} $  & $  \frac{-1}{1920} $  & $  \frac{-3}{640} $  & $  0 $  & $  \frac{3}{640} $  & $  0 $  & $  0 $  & $  0 $  & $  \frac{-1}{8} $  \\ \hline 
$ q^{2}{\mathcal F}(4\tau)                           \theta^{2}(\tau)                   \Psi(\tau)  \Psi(3\tau) \Psi^{2}(6\tau) $ & $  \frac{1}{1920} $  & $  \frac{-1}{1920} $  & $  \frac{-3}{640} $  & $  0 $  & $  \frac{3}{640} $  & $  0 $  & $  \frac{-1}{2} $  & $  -1 $  & $  \frac{3}{8} $  \\ \hline 
$ q^{2}{\mathcal F}(4\tau)                           \theta^{2}(3\tau)                  \Psi(\tau)  \Psi^{2}(2\tau) \Psi(3\tau) $ & $  \frac{1}{1920} $  & $  \frac{-1}{1920} $  & $  \frac{-3}{640} $  & $  0 $  & $  \frac{3}{640} $  & $  0 $  & $  \frac{1}{2} $  & $  1 $  & $  \frac{3}{8} $  \\ \hline 
$ {\mathcal F}(\tau)  {\mathcal F}(2\tau)  \theta^{4}(3\tau)                                                                    $ & $  \frac{1}{300} $  & $  \frac{-1}{200} $  & $  \frac{13}{100} $  & $  \frac{2}{75} $  & $  \frac{-39}{200} $  & $  \frac{26}{25} $  & $  \frac{16}{5} $  & $  \frac{32}{5} $  & $  2 $  \\ \hline 
$ {\mathcal F}(\tau)  {\mathcal F}(2\tau)  \theta^{2}(\tau) \theta^{2}(3\tau)                                                   $ & $  \frac{1}{60} $  & $  \frac{-1}{120} $  & $  \frac{-3}{20} $  & $  \frac{-2}{15} $  & $  \frac{3}{40} $  & $  \frac{6}{5} $  & $  0 $  & $  0 $  & $  6 $  \\ \hline 
$ {\mathcal F}(\tau)  {\mathcal F}(2\tau)  \theta^{4}(\tau)                                                                     $ & $  \frac{13}{300} $  & $  \frac{-13}{200} $  & $  \frac{9}{100} $  & $  \frac{26}{75} $  & $  \frac{-27}{200} $  & $  \frac{18}{25} $  & $  \frac{48}{5} $  & $  \frac{96}{5} $  & $  -6 $  \\ \hline 
$ q{\mathcal F}(\tau)  {\mathcal F}(2\tau)                                     \Psi^{2}(\tau) \Psi^{2}(3\tau)                   $ & $  \frac{1}{240} $  & $  \frac{-1}{240} $  & $  \frac{-3}{80} $  & $  0 $  & $  \frac{3}{80} $  & $  0 $  & $  0 $  & $  0 $  & $  0 $  \\ \hline 
$ q^{2}{\mathcal F}(\tau)  {\mathcal F}(2\tau)                                     \Psi^{2}(2\tau) \Psi^{2}(6\tau)              $ & $  \frac{1}{640} $  & $  \frac{-19}{1920} $  & $  \frac{-9}{640} $  & $  \frac{1}{120} $  & $  \frac{57}{640} $  & $  \frac{-3}{40} $  & $  0 $  & $  0 $  & $  \frac{-3}{8} $  \\ \hline 
$ {\mathcal F}(2\tau) {\mathcal F}(4\tau) \theta^{2}(\tau) \theta^{2}(3\tau)                                                    $ & $  \frac{1}{240} $  & $  \frac{1}{240} $  & $  \frac{-3}{80} $  & $  \frac{-2}{15} $  & $  \frac{-3}{80} $  & $  \frac{6}{5} $  & $  0 $  & $  0 $  & $  3 $  \\ \hline 
$ q{\mathcal F}(2\tau) {\mathcal F}(4\tau)                                    \Psi^{2}(\tau)\Psi^{2}(3\tau)                     $ & $  \frac{1}{960} $  & $  \frac{-1}{960} $  & $  \frac{-3}{320} $  & $  0 $  & $  \frac{3}{320} $  & $  0 $  & $  0 $  & $  0 $  & $  \frac{3}{4} $  \\ \hline 
$ {\mathcal F}^{3}(2\tau)                          \theta(\tau) \theta(3\tau)                                                   $ & $  \frac{1}{120} $  & $  0 $  & $  \frac{-3}{40} $  & $  \frac{-2}{15} $  & $  0 $  & $  \frac{6}{5} $  & $  0 $  & $  0 $  & $  0 $  \\ \hline 
$ q{\mathcal F}^{3}(2\tau)                                                              \Psi(2\tau) \Psi(6\tau)                 $ & $ \frac{1}{240}$  & $ \frac{-3}{80}$  & $ \frac{-3}{80}$  & $ \frac{1}{30}$  & $ \frac{27}{80}$  & $\frac{-3}{10}$  & $ 0$  & $ 0$  & $ 0$  \\  \hline
$ {\mathcal F}^{3}(4\tau)                        \theta(\tau) \theta(3\tau)                                                     $ & $  \frac{1}{1200} $  & $  \frac{-3}{400} $  & $  \frac{3}{400} $  & $  \frac{8}{75} $  & $  \frac{-27}{400} $  & $  \frac{24}{25} $  & $  \frac{9}{5} $  & $  \frac{18}{5} $  & $  0 $  \\ \hline 
$ {\mathcal F}^{2}(\tau) {\mathcal F}(2\tau) \theta(\tau) \theta(3\tau)                                                         $ & $  \frac{1}{30} $  & $  \frac{-1}{40} $  & $  \frac{-3}{10} $  & $  \frac{-2}{15} $  & $  \frac{9}{40} $  & $  \frac{6}{5} $  & $  0 $  & $  0 $  & $  6 $  \\ \hline 
$ q{\mathcal F}^{2}(\tau) {\mathcal F}(2\tau)                                       \Psi(2\tau) \Psi(6\tau)                     $ & $  \frac{1}{96} $  & $  \frac{-7}{160} $  & $  \frac{-3}{32} $  & $  \frac{1}{30} $  & $  \frac{63}{160} $  & $  \frac{-3}{10} $  & $  0 $  & $  0 $  & $  \frac{-3}{2} $  \\ \hline 
$ {\mathcal F}(\tau) {\mathcal F}^{2}(2\tau)  \theta(\tau) \theta(3\tau)                                                        $ & $  \frac{1}{75} $  & $  \frac{-1}{50} $  & $  \frac{3}{25} $  & $  \frac{8}{75} $  & $  \frac{-9}{50} $  & $  \frac{24}{25} $  & $  \frac{24}{5} $  & $  \frac{48}{5} $  & $  0 $  \\ \hline 
$ {\mathcal F}(2\tau){\mathcal F}^{2}(4\tau)  \theta(\tau) \theta(3\tau)                                                        $ & $  \frac{1}{480} $  & $  \frac{1}{160} $  & $  \frac{-3}{160} $  & $  \frac{-2}{15} $  & $  \frac{-9}{160} $  & $  \frac{6}{5} $  & $  0 $  & $  0 $  & $  \frac{3}{2} $  \\ \hline 
$ {\mathcal F}(\tau) {\mathcal F}(2\tau) {\mathcal F}(4\tau) \theta(\tau) \theta(3\tau)                                         $ & $  \frac{1}{120} $  & $  0 $  & $  \frac{-3}{40} $  & $  \frac{-2}{15} $  & $  0 $  & $  \frac{6}{5} $  & $  0 $  & $  0 $  & $  6 $  \\ \hline 
$ q{\mathcal F}(\tau) {\mathcal F}(2\tau) {\mathcal F}(4\tau) \Psi(2\tau) \Psi(6\tau)                                           $ & $  \frac{1}{960} $  & $  \frac{1}{64} $  & $  \frac{-3}{320} $  & $  \frac{-1}{60} $  & $  \frac{-9}{64} $  & $  \frac{3}{20} $  & $  0 $  & $  0 $  & $  \frac{3}{4} $  \\ \hline 

\end{tabularx}
\end{center}

By comparing the $n$-th Fourier coefficients of the expressions in \eqref{21mixed}, we obtain the following formulas: 

\begin{equation*}
\begin{split}
{\mathcal N}_{l,s}(2^{1}             ; 1^{3} 3^{3}                     ; n   ) &  = 2\sigma_{3}(n)  -  18\sigma_{3}(n/3) -  32\sigma_{3}(n/4)  +  288\sigma_{3}(n/12)   +  4a_{4,12}(n), \\
{\mathcal N}_{l,s,t}( 2^{1}           ; 1^{2} 3^{2} ; 2^{1} 6^{1}      ; n-1 ) &  = \frac{1}{2}\sigma_{3}(n) -    \frac{1}{2}\sigma_{3}(n/2) -    \frac{9}{2}\sigma_{3}(n/3)  +  \frac{9}{2}\sigma_{3}(n/6)    +  \frac{1}{2}a_{4,12}(n), \\
{\mathcal N}_{l,s,t}( 2^{1}           ; 1^{1} 3^{1} ; 2^{2} 6^{2}      ; n-2 ) &  = \frac{1}{8}\sigma_{3}(n) -    \frac{1}{8}\sigma_{3}(n/2) -    \frac{9}{8}\sigma_{3}(n/3)  +  \frac{9}{8}\sigma_{3}(n/6)    -    \frac{1}{8}a_{4,12}(n), \\
{\mathcal N}_{l,s,t}( 4^{1}           ; 1^{2}       ; 1^{1} 3^{1} 6^{2}; n-2 ) &  = \frac{1}{8}\sigma_{3}(n) -    \frac{1}{8}\sigma_{3}(n/2) -    \frac{9}{8}\sigma_{3}(n/3)  +  \frac{9}{8}\sigma_{3}(n/6)  -    \frac{1}{2}a_{4,6}(n)\\
& \quad  -  a_{4,6}(n/2) +  \frac{3}{8}a_{4,12}(n), \\
{\mathcal N}_{l,s,t}( 4^{1}           ; 3^{2}       ; 1^{1} 2^{2} 3^{1}; n-1 ) &  = \frac{1}{8}\sigma_{3}(n) -    \frac{1}{8}\sigma_{3}(n/2) -    \frac{9}{8}\sigma_{3}(n/3)  +  \frac{9}{8}\sigma_{3}(n/6)  +  \frac{1}{2}a_{4,6}(n)\\ 
& \quad +  a_{4,6}(n/2) +  \frac{3}{8}a_{4,12}(n), \\
{\mathcal N}_{l,s}( 1^{1} 2^{1}       ; 3^{4}                          ; n   ) &  = \frac{4}{5}\sigma_{3}(n) -    \frac{6}{5}\sigma_{3}(n/2) +  \frac{156}{5}\sigma_{3}(n/3) +  \frac{32}{5}\sigma_{3}(n/4) -    \frac{234}{5}\sigma_{3}(n/6) \\ 
& \quad +  \frac{1248}{5}\sigma_{3}(n/12)   + \frac{16}{5}a_{4,6}(n) +  \frac{32}{5}a_{4,6}(n/2) +  2a_{4,12}(n), \\
\end{split}
\end{equation*}
\begin{equation*}
\begin{split}
{\mathcal N}_{l,s}( 1^{1} 2^{1}       ; 1^{2} 3^{2}                    ; n   ) &  = 4\sigma_{3}(n) -  2\sigma_{3}(n/2) -  36\sigma_{3}(n/3) -  32\sigma_{3}(n/4) +  18\sigma_{3}(n/6) \\& \quad +  288\sigma_{3}(n/12)  +  6a_{4,12}(n), \\
{\mathcal N}_{l,s}( 1^{1} 2^{1}       ; 1^{4}                          ; n  )  &  =  \frac{52}{5}\sigma_{3}(n) -    \frac{78}{5}\sigma_{3}(n/2) +  \frac{108}{5}\sigma_{3}(n/3) +  \frac{416}{5}\sigma_{3}(n/4) -    \frac{162}{5}\sigma_{3}(n/6)  \\ 
& \quad + \frac{864}{5}\sigma_{3}(n/12) +   \frac{48}{5}a_{4,6}(n) +  \frac{96}{5}a_{4,6}(n/2) -  6a_{4,12}(n), \\
{\mathcal N}_{l,t}( 1^{1} 2^{1}                     ; 1^{2} 3^{2}      ; n-1 ) &  = \sigma_{3}(n) -  \sigma_{3}(n/2) -  9\sigma_{3}(n/3)  +  9\sigma_{3}(n/6),     \\
{\mathcal N}_{l,t}( 1^{1} 2^{1}                     ; 2^{2} 6^{2}      ; n-2 ) &  = \frac{3}{8}\sigma_{3}(n) -    \frac{19}{8}\sigma_{3}(n/2) -    \frac{27}{8}\sigma_{3}(n/3) +  2\sigma_{3}(n/4) +  \frac{171}{8}\sigma_{3}(n/6) \\ 
& \quad -  18\sigma_{3}(n/12)  -    \frac{3}{8}a_{4,12}(n), \\
{\mathcal N}_{l,s}( 2^{1} 4^{1}       ; 1^{2} 3^{2}                    ; n  )  &  =  \sigma_{3}(n) +  \sigma_{3}(n/2) -  9\sigma_{3}(n/3) -  32\sigma_{3}(n/4) -  9\sigma_{3}(n/6) \\& \quad +  288\sigma_{3}(n/12)   +  3a_{4,12}(n), \\
{\mathcal N}_{l,t}( 2^{1} 4^{1}                     ; 1^{2}.3^{2}      ; n-1)  &  =  \frac{1}{4}\sigma_{3}(n) -    \frac{1}{4}\sigma_{3}(n/2) -    \frac{9}{4}\sigma_{3}(n/3)  +  \frac{9}{4}\sigma_{3}(n/6)    +  \frac{3}{4}a_{4,12}(n), \\
{\mathcal N}_{l,s}( 2^{3}             ; 1^{1} 3^{1}                    ; n  )  &  =  2\sigma_{3}(n)  -  18\sigma_{3}(n/3) -  32\sigma_{3}(n/4)  +  288\sigma_{3}(n/12),    \\
{\mathcal N}_{l,t}( 2^{3}                           ; 2^{1} 6^{1}      ; n-1)  &  =  \sigma_{3}(n) -  9\sigma_{3}(n/2) -  9\sigma_{3}(n/3) +  8\sigma_{3}(n/4) +  81\sigma_{3}(n/6) -  72\sigma_{3}(n/12),    \\
{\mathcal N}_{l,s}( 4^{3}             ; 1^{1} 3^{1}                    ; n  )  &  =  \frac{1}{5}\sigma_{3}(n) -    \frac{9}{5}\sigma_{3}(n/2) +  \frac{9}{5}\sigma_{3}(n/3) +  \frac{128}{5}\sigma_{3}(n/4) -    \frac{81}{5}\sigma_{3}(n/6) \\ 
& \quad +  \frac{1152}{5}\sigma_{3}(n/12)  +  \frac{9}{5}a_{4,6}(n) +  \frac{18}{5}a_{4,6}(n/2),  \\
{\mathcal N}_{l,s}( 1^{2} 2^{1}       ; 1^{1} 3^{1}                    ; n  )  &  =  8\sigma_{3}(n) -  6\sigma_{3}(n/2) -  72\sigma_{3}(n/3) -  32\sigma_{3}(n/4) +  54\sigma_{3}(n/6) \\& \quad +  288\sigma_{3}(n/12)    +  6a_{4,12}(n), \\
{\mathcal N}_{l,t}( 1^{2} 2^{1}                     ; 2^{1} 6^{1}      ; n-1)  &  =  \frac{5}{2}\sigma_{3}(n) -    \frac{21}{2}\sigma_{3}(n/2) -    \frac{45}{2}\sigma_{3}(n/3) +  8\sigma_{3}(n/4) +  \frac{189}{2}\sigma_{3}(n/6)\\ 
& \quad  -  72\sigma_{3}(n/12)   -    \frac{3}{2}a_{4,12}(n), \\
{\mathcal N}_{l,s}( 1^{1} 2^{2}       ; 1^{1} 3^{1}                    ; n  )  &  =  \frac{16}{5}\sigma_{3}(n) -    \frac{24}{5}\sigma_{3}(n/2) +  \frac{144}{5}\sigma_{3}(n/3) +  \frac{128}{5}\sigma_{3}(n/4) -    \frac{216}{5}\sigma_{3}(n/6)\\ 
& \quad +  \frac{1152}{5}\sigma_{3}(n/12) +  \frac{24}{5}a_{4,6}(n) +  \frac{48}{5}a_{4,6}(n/2),  \\
{\mathcal N}_{l,s}( 2^{1} 4^{2}       ; 1^{1} 3^{1}                    ; n  )  &  =  \frac{1}{2}\sigma_{3}(n) +  \frac{3}{2}\sigma_{3}(n/2) -    \frac{9}{2}\sigma_{3}(n/3) -  32\sigma_{3}(n/4) -    \frac{27}{2}\sigma_{3}(n/6)\\
& \quad +  288\sigma_{3}(n/12)   +  \frac{3}{2}a_{4,12}(n), \\
{\mathcal N}_{l,s}(1^{1} 2^{1} 4^{1} ; 1^{1} 3^{1}                     ; n   ) &  = 2\sigma_{3}(n)  -  18\sigma_{3}(n/3) -  32\sigma_{3}(n/4)  +  288\sigma_{3}(n/12)   +  6a_{4,12}(n), \\
{\mathcal N}_{l,t}( 1^{1} 2^{1}.4^{1}                ; 2^{1} 6^{1}     ; n-1)  &  =  \frac{1}{4}\sigma_{3}(n) +  \frac{15}{4}\sigma_{3}(n/2) -    \frac{9}{4}\sigma_{3}(n/3) -  4\sigma_{3}(n/4) \\
& \quad -    \frac{135}{4}\sigma_{3}(n/6) +  36\sigma_{3}(n/12)   +  \frac{3}{4}a_{4,12}(n). \\
\end{split}
\end{equation*}
\begin{rmk}
The above formulas are exactly the same as in Theorem 1.1 of \cite{xia}. 
The space of cusp forms $S_4(\Gamma_0(12))$ is three dimensional spanned by $f_{4,6}(\tau), f_{4,6}(2\tau)$ and $f_{4,12}(\tau)$. 
However, in \cite[Theorem 3.1]{xia}, they make use of only two cusp forms for getting the required expressions. We observe that the two cusp 
forms appear in \cite[Theorem 3.1]{xia}, denoted by $G(\tau)$ and $H(\tau)$, can be given in terms of our basis elements as follows.

\begin{equation}\label{gh}
\begin{split}
G(\tau)& =  -\frac{1}{6} f_{4,6}(\tau) -\frac{1}{3} f_{4,6}(2\tau) + \frac{1}{6} f_{4,12}(\tau) \\
H(\tau)& =  \frac{1}{2} f_{4,6}(\tau) + f_{4,6}(2\tau) + \frac{1}{2} f_{4,12}(\tau). 
\end{split}
\end{equation}
\end{rmk}


We now make the following observation about obtaining representations of $E_4(\tau)$ and its duplications in terms of $p,k$ using our identities 
\eqref{21mixed}. \\

To get the identities given by \eqref{21mixed} (with $G(\tau), H(\tau)$ in place of $f_{4,6}(\tau), f_{4,6}(2\tau), f_{4,12}(\tau)$), Xia et. al used the method of $(p,k)$-parametrization. 
Let 
\begin{equation*}
 \begin{split}
  p & =  p(\tau): =  \frac{\theta^2(\tau)-\theta^2(3\tau)}{2 \theta^2(3\tau)},{~~~~~~~~~~~~~~}
  k  = k(\tau): =  \frac{\theta^3(3\tau)}{\theta(\tau)}. \\
 \end{split}
\end{equation*}
In \cite[Theorems 1,2,4]{aaw2}, $(p,k)$-parametrizations of ${\mathcal F}(d\tau)$, $d=1,2,4$ are given and in \cite[(2.3)]{aalw1} the parametrization is 
given for the theta series $\theta(\tau)$ and $\theta(3\tau)$. Further, in \cite{aaw1}, the representations of $q^{j/24}\prod_{n=1}^\infty(1-q^{nj})$, 
$j=1,2,3,4,6,12$ in terms of $p$ and $k$ have been established. In \cite{aw}, Alaca and Williams derived the representations of $E_4(d\tau)$, $d=1,2,3,4,6,12$ in terms of $p$ and $k$. In \cite{xia}, all the above $(p,k)$-parametrizations were used to get the identity \eqref{21mixed} for all 
the 21 mixed forms. Since we establish \eqref{21mixed} using the theory of modular forms, we use these identities along with the $(p,k)$-parametrizations obtained in   \cite[Theorems 1,2,4]{aaw2}, \cite[(2.3)]{aalw1} and \cite{aaw1} (for the generating functions ${\mathcal F}(\tau)$, 
$\theta(\tau)$ and $\Psi(\tau)$) to derive the representations of $E_4(\tau)$ and its duplications in terms of $p$ and $k$. 
Using the $(p,k)$-parametrization for the left-hand side generating functions in \eqref{21mixed}, we get a system of equations involving 
$E_4(d\tau)$, $d\vert 12$, $f_{4,6}(\ell\tau)$, $\ell =1,2$ and $f_{4,12}(\tau)$ and the functions $p$ and $k$, from which we derive the 
required representations of $E_4(d\tau)$, $d\vert 12$, $f_{4,6}(\ell\tau)$, $\ell =1,2$ and $f_{4,12}(\tau)$ in terms of $p, k$, which we give below. (We have used Mathematica software for doing these computations.) 

\begin{equation*}
 \begin{split}
E_{4}(\tau) & =  (1 + 124 p + 964 p^{2} + 2788 p^{3} + 3910 p^{4} + 2788 p^{5} + 964 p^{6} + 124 p^{7} + p^{8}       ) k^{4}, \\
E_{4}(2\tau) & =  (1 + 4 p + 64 p^{2} + 178 p^{3} + 235 p^{4} + 178 p^{5} + 64 p^{6} + 4 p^{7} + p^{8}       ) k^{4}, \\
E_{4}(3\tau) & =  (1 + 4 p + 4 p^{2} + 28 p^{3} + 70 p^{4} + 28 p^{5} + 4 p^{6} + 4 p^{7} + p^{8}          ) k^{4}, \\
E_{4}(4\tau) & =  (1 + 4 p + 4 p^{2} - 2 p^{3} + 10 p^{4} + 28 p^{5} + \frac{31}{4} p^{6} - \frac{29}{4} p^{7} + \frac{1}{16}p^{8}) k^{4}, \\
E_{4}(6\tau) & =  (1 + 4 p + 4 p^{2} - 2 p^{3} - 5 p^{4} - 2 p^{5} + 4 p^{6} + 4 p^{7} + p^{8}             ) k^{4},  \\
E_{4}(12\tau) & =  (1 + 4 p + 4 p^{2} - 2 p^{3} - 5 p^{4} - 2 p^{5} + \frac{1}{4} p^{6} + \frac{1}{4} p^{7} + \frac{1}{16}p^{8}         ) k^{4}, \\
f_{4,6}(\tau) & = (-1 - 4 p -\frac{119}{32}  p^{2} + \frac{115}{32} p^{3} - \frac{913}{128} p^{4} - \frac{1695}{64} p^{5} - \frac{2049}{256} p^{6} + \frac{1801}{256} p^{7} -\frac{1}{16} p^{8}) k^{4},\\
f_{4,6}(2\tau) & =(\frac{1}{2} + \frac{9}{4} p + \frac{175}{64}p^{2} -\frac{83}{64}  p^{3} +\frac{673}{256} p^{4} + \frac{1583}{128} p^{5} + \frac{2081}{512} p^{6} - \frac{1737}{512} p^{7} +\frac{1}{32} p^{8})k^{4},\\
f_{4,12}(\tau) & =(\frac{1}{2} p + \frac{7}{4} p^{2} + \frac{7}{4} p^{3} - \frac{7}{4} p^{5} - \frac{7}{4} p^{6} -\frac{1}{2} p^{7})k^{4}. \\
\end{split}
\end{equation*}
Now using \eqref{gh}, we obtain representations of $G(\tau)$ and $H(\tau)$ in terms of $p,k$, which is given below. \\

\begin{equation*}
\begin{split}
G(\tau) & =   \frac{2 p^{3} + 5 p^{4} - 5 p^{6} - 2 p^{7}}{16} k^{4} ~=~  \frac{p^3(1-p)(1+ p)(1+2p)(2+p)}{16} k^{4},\\
H(\tau) & = \frac{8 p + 28 p^{2} + 22 p^{3} - 15 p^{4} - 28 p^{5} - 13 p^{6} - 2 p^{7}}{16} k^{4}~=~ \frac{p(1-p)(1+p)(1+2p)(2+p)^3}{16} k^{4}. \\
\end{split}
\end{equation*}
\smallskip

\section{Explicit formulas in the case of 4 and 6 variables}

In this section, we give a list of forms in $4$ and $6$ variables for the purpose of providing explicit examples for our main results. 
Since the number of examples is too big, we present a few sample formulas in each case and the other formulas are expressed in terms 
of linear combination of basis elements (of the respective vector space of modular forms) and the linear combination coefficients are listed in  
tabular format at the end of this paper. \\

\smallskip

\centerline{\bf {Table 2(a). Forms with 4 variables}}

\bigskip

Recall that $T_{\mathcal C}$ denotes the sum of triangular numbers with coefficients in ${\mathcal C}$, given by \eqref{triang} and 
${\mathcal M}_{s,t}$, ${\mathcal M}_{l,t}$ are the mixed forms given by the equation \eqref{st-lt}. \\
{\tiny
\begin{center}


\smallskip

\begin{tabular}{|p{1.5 cm }|p{2.7cm}|p{6.7 cm }|p{4.75cm}|}

\hline

{\textbf{Space}} & \multicolumn{3}{|c|}{\textbf{Form}} \\ \cline{2-4} 
&&&\\
& \textbf{$T_{\mathcal C}$:$(1^{e_1}2^{e_2}3^{e_3}4^{e_4}6^{e_5})$} & {\bf${\mathcal M}_{s,t}$:$(1^{j_1}2^{j_2}3^{j_1}6^{j_4};1^{e_1}2^{e_2}3^{e_1}4^{e_4}6^{e_5})$} & \textbf{${\mathcal M}_{l,t}$:$(1^{i_1}2^{i_2}4^{i_1}8^{i_4};1^{e_1}2^{e_2}3^{e_1}4^{e_4}6^{e_5})$}  \\

\hline
&&&\\
 $M_{2}(6,\chi_{0})$& $(1^{2}3^{2})$ & &  \\

\hline
&&&\\
$M_{2}(8,\chi_{0})$ & $(2^{4})$& $(1^{2} ; 4^{2} ),(2^{2} ; 4^{2} ),(2^{1} ; 2^{2}4^{1} )$ &   \\

\hline
&&&\\
\hline

\end{tabular}

\end{center}
}

\bigskip

\centerline{\bf {Table 2(b). Forms with 6 variables}}

\smallskip

Apart from the three types as in Table 2(a), here we also consider the mixed form ${\mathcal M}$ given by \eqref{mixed}, having all three types 
of forms. 

{\tiny
\begin{center}
\begin{tabular}{|c|c|c|}

\hline
{\textbf{Space}} & \multicolumn{2}{|c|}{\textbf{Form}} \\ \cline{2-3} 
&&\\
& {\bf $T_{\mathcal C}$:$(1^{e_1}2^{e_2}3^{e_3}4^{e_4}6^{e_5})$} & {\bf ${\mathcal M}_{l,t}$ :($1^{i_1}2^{i_2}4^{i_3}8^{i_4};1^{e_1}2^{e_2}3^{e_3}4^{e_4}6^{e_5}$)} \\
\hline 
&&\\                                                                                                                
$M_{3}(4,\chi_{-4})$ & $(1^{4} 2^{2})$ &   \\                                                                                                                
\hline                                                                                                                

&&\\                                                                                                                
$M_{3}(8,\chi_{-4})$ & $(2^{4} 4^{2}), (4^{6}) $ &   \\                                                                                                                
\hline                                                                                                                
&&\\
\hline

\end{tabular}
\end{center}
\newpage

\begin{center}

\begin{tabularx}{\textwidth}{|p{1.5 cm}|p{8.1 cm}|p{7.2 cm}|}
\hline
     
                            \multicolumn{1}{|c|}{\textbf{Space}}  & \multicolumn{2}{|c|}{\textbf{Form}} \\ \cline{2-3} 
& & \\                            
& \textbf{${\mathcal M}_{s,t}$ :($1^{i_1}2^{i_2}3^{i_3}6^{i_4};1^{e_1}2^{e_2}3^{e_3}4^{e_4}6^{e_5}$)} & \textbf {${\mathcal M}$ : ($1^{i_1}2^{i_2}4^{i_3}8^{i_4};1^{j_1}2^{j_2}3^{j_3}6^{j_4};1^{e_1}2^{e_2}3^{e_3}4^{e_4}6^{e_5}$)} \\
\hline 
     
\endfirsthead
      \hline
                              \multicolumn{1}{|c|}{\textbf{Space}}  & \multicolumn{2}{|c|}{\textbf{Form}} \\ \cline{2-3}
& & \\                              
   & \textbf{${\mathcal M}_{s,t}$ :($1^{i_1}2^{i_2}3^{i_3}6^{i_4};1^{e_1}2^{e_2}3^{e_3}4^{e_4}6^{e_5}$)} & \textbf {${\mathcal M}$ : ($1^{i_1}2^{i_2}4^{i_3}8^{i_4};1^{j_1}2^{j_2}3^{j_3}6^{j_4};1^{e_1}2^{e_2}3^{e_3}4^{e_4}6^{e_5}$)} \\
\hline 
     
\endhead
      \hline
      \multicolumn{3}{|r|}{{Continued on Next Page\ldots}}     \\ \hline
      
\endfoot
    
\endlastfoot
\hline
&&\\
 $M_{3}(4,\chi_{-4}) $ &      $(1^{2}; 2^{4})$                                                                                                                                    &                                                                                                                              \\                                                                                                                                                                                                                                               
\hline                                                                                                                                                                                                                                                                                                     
&&\\                                                                                                                                                                                                                                                                                                      
 $M_{3}(8,\chi_{-4}) $ &      $(1^{4}; 4^{2}),(2^{4}; 4^{2}),(1^{2} 2^{2}; 4^{2}),(1^{2}; 4^{4}),(2^{2}; 4^{4}),(2^{2}; 2^{4}),$                                                  &                                                                                                                              \\
                       &      $(1^{1} 2^{1}  ; 1^{2} 2^{1} 4^{1}),(1^{2} 2^{1}; 2^{2} 4^{1}),(2^{3}; 2^{2} 4^{1})$                                                                &                                                                                                                              \\
\hline 
\end{tabularx}

\end{center}
}

\subsection{Explicit bases for the space $M_{k}(\Gamma_0(N),\chi)$, $k = 2,3$, $N=6,8,12,24$.}

In order to give explicit examples for the cases given in Tables 2 and 3, we shall provide explicit  bases for the spaces of modular forms 
$M_{k}(\Gamma_0(N), \chi)$ for $k=2,3$, and  $N=6,8,12,24$. 

\smallskip

\noindent 
The following functions are used to construct the explicit bases for the spaces of modular forms.

For an integer $k \ge 4,$ let $E_k$ denote the normalized Eisenstein series of weight $k$ for the full modular group $SL_2({\mathbb Z})$, 
given by 
$$
E_k(\tau) = 1 - \frac{2k}{B_k}\sum_{n\ge 1} \sigma_{k-1}(n) q^n,
$$
where $q=e^{2 i\pi \tau}$,$\tau \in {\mathbb H}$, $\sigma_r(n) = \sum_{d\vert n} d^{r}$ and $B_k$ is the $k$-th Bernoulli number defined by following identitty $\displaystyle{\frac{x}{e^x-1} = \sum_{m=0}^\infty \frac{B_m}{m!} x^m}$.
The above Fourier expansion is also valid for $k=2$ and is given by 
$$
E_2(\tau) = 1 - 24 \sum_{n\ge 1} \sigma(n) q^n
$$
and it is not a modular form but a quasimodular form of weight $2$. However, one can construct weight 2 modular forms using $E_2(\tau)$, 
which is defined as follows. For a given natural number $N$, let $a, b$ be positive divisors of $N$ with $a< b$. Then the following function is a 
modular form of weight $2$ on $\Gamma_0({\rm lcm}[a,b])$, with trivial character:
\begin{equation}\label{phiab}
\phi_{a,b}(\tau) := \frac{b E_{2}(b\tau)- a E_{2}(a\tau)}{b-a}. 
\end{equation}

\smallskip

Apart from these Eisenstein series, we also use the following generalized Eisenstein series. Suppose that $\chi$ and $\psi$ are primitive Dirichlet characters with conductors $M$ and $N$, respectively. For a positive integer $k$,  let 
\begin{equation}\label{eisenstein}
E_{k,\chi,\psi}(\tau) :=  - \frac{B_{k,\psi}}{2k} ~ \delta_{M,1} + \sum_{n\ge 1}\left(\sum_{d\vert n} \psi(d) \cdot \chi(n/d) d^{k-1}\right) q^n,
\end{equation}
where $\delta_{m,n}$ is the Kronecker delta function and 
$B_{k,\psi}$ is the  generalized Bernoulli number with respect to the character $\psi$. 
Then, the Eisenstein series $E_{k,\chi,\psi}(\tau)$ belongs to the space $M_k(\Gamma_0(MN), \chi \psi)$, provided $\chi(-1)\psi(-1) = (-1)^k$ 
and $MN\not=1$. When $\chi=\psi =1$ (i.e., when $M=N=1$) and $k\ge 4$, we have $E_{k,\chi,\psi}(\tau) = -\frac{B_k}{2k} E_k(\tau)$, the normalized Eisenstein series of integer weight $k$ as defined before.
For more details on the construction of these Eisenstein series, we refer to \cite{{miyake}, {stein}}. We denote the inner sum in \eqref{eisenstein} by 
$\sigma_{k-1;\chi,\psi}(n)$. i.e.,
\begin{equation}\label{divisor}
\sigma_{k-1;\chi,\psi}(n) := \sum_{d\vert n} \psi(d) \cdot \chi(n/d) d^{k-1}.
\end{equation}

\smallskip

An eta-quotient is defined as a finite product of integer powers of $\eta(\tau)$ where, $\eta(\tau)=q^{1/24} \prod_{n\ge1}(1-q^n)$ is the Dedekind eta function  and we denote it as follows. 
\begin{equation}\label{eta-q} 
\delta_{1}^{r_{\delta_{1}}} \delta_{2}^{r_{\delta_{2}}} \cdots \delta_{s}^{r_{\delta_{s}}} := \prod_{i=1}^s \eta^{r_{\delta_{i}}}(\delta_{i} \tau ),
\end{equation}
where $\delta_{i}$'s are positive integers and $r_{\delta_{i}}$'s are non-zero integers.

\bigskip

From the basic facts of the theory of modular forms, we have the following decomposition with respect to peterson inner product
$$
M_k(\Gamma_0(N), \chi) = {\mathcal E}_k(\Gamma_0(N), \chi) \oplus S_k(\Gamma_0(N), \chi),
$$
where, ${\mathcal E}_k(\Gamma_0(N),\chi)$ and $S_k(\Gamma_0(N),\chi)$ denote the the space generated by Eisenstein series and space of cusp forms respectively. For a given weight $k$, level $N$ with character $\chi$, we denote by $\nu_{k,N,\chi}$ the dimension of the vector space $M_k(\Gamma_0(N),\chi)$. We write a general basis for $M_k(\Gamma_0(N),\chi)$ as $f_{k,N,\chi;j}(\tau)$,
where $1\le j\le \nu_{k,N,\chi}$. In the basis, we first list the basis elements for the space of Eisenstein series ${\mathcal E}_k(\Gamma_0(N),\chi)$, followed by the basis elements of $S_k(\Gamma_0(N),\chi)$.
Denoting by $e_{k,N,\chi}$ and $s_{k,N,\chi}$ the dimensions of the vector spaces ${\mathcal E}_k(\Gamma_0(N),\chi)$ and $S_k(\Gamma_0(N),\chi)$ respectively, we have $\nu_{k,N,\chi} =  e_{k,N,\chi} + s_{k,N,\chi}$. 

\bigskip

The following tables give bases for the vector spaces $M_2(\Gamma_0(N),\chi)$ and $M_3(\Gamma_0(N),\chi)$. 
The character $\chi_0$ denotes the principal character modulo the respective level $N$ and $\chi_m$ denotes the Kronecker symbol $\left(\frac{m}{\cdot}\right)$. 


\begin{center}

\textbf{Table A. Basis for $M_{2}(\Gamma_0(N),\chi)$}\\

\bigskip

\begin{tabular}{|p{1.75 cm }|p{0.8 cm}|p{0.7 cm }|p{6cm}|p{4.75 cm }|}

\hline

{\textbf{ Space}} & \multicolumn{2}{|c|}{\textbf{Dimension}} & 
 \multicolumn{2}{|c|}{\textbf{Basis for  $M_2(\Gamma_0(N),\chi)$}}\\ \cline{2-5}
&&&&\\
 {$(N,\chi)$}     & $e_{2,N,\chi}$ & $s_{2,N,\chi}$  & {Basis for $ {\mathcal E}_2(\Gamma_0(N),\chi)$ }& {Basis for $S_2(\Gamma_0(N),\chi)$} \\
 
\hline 
 $(6,\chi_{0})$   & 3   &  0   &  $\{ \phi_{1,b}, b\vert 6, b \neq 1\}$               &  \{0\}  \\
                           
\hline                     
 $(8,\chi_{0})$   &  3  &  0   &   $\{ \phi_{1,b}, b|8, b \neq 1 \}$               & \{0\} \\
                           
\hline                     
 $(8,\chi_{8})$   &  2  &  0   & $\{E_{2,{\bf 1},\chi_{8}}(\tau)$, $ E_{2,\chi_{8},{\bf 1}}(\tau)\}$ & \{0\} \\
                           
\hline                     
 $(12,\chi_{0})$  &  5  &  0   & $\{ \phi_{1,b}, b|12, b \neq 1\}$   & \{0\}  \\
                                    
\hline                     
 $(12,\chi_{12})$ &  4  &  0  & $\{E_{2, {\bf 1}, \chi_{12}}(\tau),  E_{2, \chi_{12}, {\bf 1}}(\tau), $ $ E_{2, \chi_{4}, \chi_{3}}(\tau),  E_{2, \chi_{3}, \chi_{4}}(\tau)\}$ & \{0\} \\

\hline                     
 $(24,\chi_{0})$  &  7  &  1   & $\{ \phi_{1,b}, b|24, b \neq 1\}$ & $\{\Delta_{2, 24, \chi_{0}}(\tau) \}$ \\
                           
\hline                     
 $(24,\chi_{8})$  &  4  &  2   & $\{E_{2, {\bf 1}, \chi_{8}}(az), a|3;$ $E_{2, \chi_{8}, {\bf 1}}(bz), b|3\}$ & $\{\Delta_{2, 24, \chi_{8};1}(\tau),  
 \Delta_{2, 24, \chi_{8};2}(\tau)\}$ \\
                           
 \hline                     
 $(24,\chi_{12})$ &  8  &  0  & $\{E_{2, {\bf 1}, \chi_{12}}(az), a|2;  E_{2, \chi_{12}, {\bf 1}}(bz), b|2;$ $E_{2, \chi_{4}, \chi_{3}}(cz), c|2;  
 E_{2, \chi_{3}, \chi_{4}}(dz), d|2\}$ & \{0\} \\
                           
\hline                     
 $(24,\chi_{24})$ &  4  &  2   & $\{E_{2, {\bf 1}, \chi_{24}}(\tau),  E_{2, \chi_{24}, {\bf 1}}(\tau), $ $ E_{2, \chi_{8}, \chi_{3}}(\tau),  E_{2, \chi_{3}, \chi_{8}}(\tau)\}$ & $\{\Delta_{2, 24, \chi_{24};1}(\tau), \Delta_{2, 24, \chi_{24};2}(\tau)\}$ \\
 
\hline 
\end{tabular}
\end{center}
\smallskip

\begin{center}

\textbf{Table B. Basis for $M_{3}(\Gamma_0(N),\chi)$}

\smallskip

\begin{tabular}{|p{1.5 cm }|p{0.7 cm}|p{0.7 cm }|p{5.9cm}|p{5.9 cm }|}

\hline

{\textbf{Space}} & \multicolumn{2}{|c|}{\textbf{Dimension}} & 
 \multicolumn{2}{|c|}{\textbf{Basis for  $M_3(\Gamma_0(N),\chi)$}}\\ \cline{2-5}
 &&&&\\
 {$(N,\chi)$} & $e_{3,N,\chi}$ & $s_{3,N,\chi}$  & {Basis for $ {\mathcal E}_3(\Gamma_0(N),\chi)$ }& {Basis for $S_3(\Gamma_0(N),\chi)$}
  \\ \hline 
 $(4,\chi_{-4})$    &   2 & 0       & $\{E_{3, {\bf 1}, \chi_{-4}}(\tau),  E_{3, \chi_{-4}, {\bf 1}}(\tau)\}$                   &  $\{0\}$  \\
\hline                                                                                                                   
 $(3,\chi_{-3})$    &   2 & 0       & $\{E_{3, {\bf 1}, \chi_{-3}}(\tau),  E_{3, \chi_{-3}, {\bf 1}}(\tau)\}$                   & $\{0\}$ \\
                                                                                                                         
 \hline                                                                                                                  
 $(6,\chi_{-3})$    &   4 & 0       & $\{E_{3, {\bf 1}, \chi_{-3}}(az), a|2;   E_{3, \chi_{-3}, {\bf 1}}(bz), b|2\}$      & $\{0\}$ \\
\hline                                                                                                                   
$(8,\chi_{-4})$     &   4 & 0       & $\{E_{3, {\bf 1}, \chi_{-4}}(az), a|2; $ $ E_{3, \chi_{-4}, {\bf 1}}(bz), b|2\}$    & $\{0\}$  \\
                                                                                                                         
\hline                              
$(8,\chi_{-8})$     &   2 & 1       & $\{E_{3, {\bf 1}, \chi_{-8}}(\tau),$ $ E_{3, \chi_{-8}, {\bf 1}}(\tau)\}$                 & $\{\Delta_{3, 8, \chi_{-8}}(\tau)\}$ \\
                                    
\hline                              
$(12,\chi_{-3})$    &   6 & 1       & $\{E_{3, {\bf 1}, \chi_{-3}}(az), a|4; $ $ E_{3, \chi_{-3}, {\bf 1}}(bz), b|4\}$    & $\{\Delta_{3, 12, \chi_{-3}}(\tau) \}$ \\
                                                                                                                          
\hline                                                                                                                    
$(12,\chi_{-4})$    &   4 & 2       & $\{E_{3, {\bf 1}, \chi_{-4}}(az), a|3; $ $ E_{3, \chi_{-4}, {\bf 1}}(bz), b|3\}$    & $\{ \Delta_{3, 12, \chi_{-4};1}(\tau),\Delta_{3, 12, \chi_{-4};2}(\tau) \}$ \\
                                                                                                                          
\hline                                                                                                                    
$(24,\chi_{-3})$    &   8 & 4       & $\{E_{3, {\bf 1}, \chi_{-3}}(az), a|8;$ $ E_{3, \chi_{-3}, {\bf 1}}(bz), b|8\}$    & $\{\Delta_{3, 12, \chi_{-3}}(az), a|2; \Delta_{3, 24, \chi_{-3};s}(\tau)$, \\
\hline                                                                                                                    
$(24,\chi_{-4})$    &   8 & 4       & $\{E_{3, {\bf 1}, \chi_{-4}}(az), a|6;$ $  E_{3, \chi_{-4}, {\bf 1}}(bz), b|6\}$    & $\{ \Delta_{3, 12, \chi_{-4};1}(az), a|2$; 
$\Delta_{3, 12, \chi_{-4};2}(bz), b|2 \}$ \\
\hline                                                                                                                    
$(24,\chi_{-8})$    &   4 & 6       & $\{E_{3, {\bf 1}, \chi_{-8}}(az); a|3,$ $  E_{3, \chi_{-8}, {\bf 1}}(bz); b|3\}$    & $\{\Delta_{3, 8, \chi_{-8}}(az), a|3; $ $   \Delta_{3, 24, \chi_{-8};s}(\tau); 1 \leq s \leq 4 \}$ \\
\hline                                                                                                                    
$(24,\chi_{-24})$   &   4 & 6       & $\{E_{3, {\bf 1}, \chi_{-24}}(\tau),  E_{3, \chi_{-24}, {\bf 1}}(\tau), $ $ E_{3, \chi_{-3}, \chi_{8}}(\tau),  E_{3, \chi_{8}, \chi_{-3}}(\tau)\}$ & $\{\Delta_{3, 24, \chi_{-24};r}(\tau); 1 \leq r \leq 6 \}$ \\
\hline 
\end{tabular}
\end{center}

%

\bigskip

\newpage

In the above tables, we use the following eta-quotients for the basis of cusp forms.

\smallskip

\begin{center}
\begin{tabular}{ll}
$\Delta_{2,24,\chi_{0}}(\tau) = 2^{1}4^{1}6^{1}12^{1}$ ,  &\\
$\Delta_{2, 24, \chi_{8};1}(\tau) = 1^{1}2^{-1}3^{-1}6^{4}8^{2}12^{-1}$ , & $\Delta_{2, 24, \chi_{8};2}(\tau)   =  1^{2}4^{-1}6^{-1}8^{1}12^{4}24^{-1}$ ,  \\
$\Delta_{2, 24, \chi_{24};1}(\tau) =  1^{1}2^{-1}3^{-1}4^{1}6^{4}12^{-2}24^{2}$ , & $ \Delta_{2, 24, \chi_{24};2}(\tau)   =  1^{2}2^{-2}4^{4}6^{1}8^{-1}12^{-1}24^{1} $ ,   \\ 
 $ \Delta_{3,8,\chi_{-8}}(\tau)        =  1^2 2^1 4^{1} 8^2 ,                               $ & $ \Delta_{3,12,\chi_{-3}}(\tau)       =  2^{3}6^{3} ,                                     $  \\
 $ \Delta_{3,12,\chi_{-4};1}(\tau)     =  1^{4}2^{-1}4^{1}6^{1}12^{1} ,                     $ & $ \Delta_{3,12,\chi_{-4};2}(\tau)     =  1^{1}2^{1}3^{1}6^{-1}12^{4},                     $  \\                                                                  
 $ \Delta_{3,24,\chi_{-3};1}(\tau)     =  1^{3}2^{1}3^{-1}4^{4}6^{1}8^{-3}24^{1} ,          $ & $ \Delta_{3,24,\chi_{-3};2}(\tau)     =  1^{-3}2^{4}3^{1}4^{1}8^{3}12^{1}24^{-1} ,        $  \\                      
 $ \Delta_{3, 24, \chi_{-8};1}(\tau)   =  1^{-2}2^{4}4^{4}6^{1}8^{-2}12^{1} ,               $ & $ \Delta_{3, 24, \chi_{-8};2}(\tau)   =  1^{2}4^{3}6^{3}8^{-1}12^{-2}24^{1},              $  \\
 $ \Delta_{3, 24, \chi_{-8};3}(\tau)   = 2^{3}3^{2}4^{-2}8^{1}12^{3}24^{-1},                $ & $ \Delta_{3, 24, \chi_{-8};4}(\tau)   =  1^{1}2^{1}3^{-1}4^{1}6^{2}8^{1}12^{2}24^{-1},    $  \\         
 $ \Delta_{3, 24, \chi_{-24};1}(\tau)  =  1^{-3}2^{9}3^{-1}4^{-3}6^{4}12^{-2}24^{2} ,       $ & $ \Delta_{3, 24, \chi_{-24};2}(\tau)   =  1^{-2}2^{8}6^{1}8^{-1}12^{-1}24^{1},            $  \\
 $ \Delta_{3, 24, \chi_{-24};3}(\tau)  =  1^{1}2^{-5}3^{-1}4^{11}6^{4}8^{-4}12^{-2}24^{2} , $ & $ \Delta_{3, 24, \chi_{-24};4}(\tau)   =  1^{2}2^{-6}4^{14}6^{1}8^{-5}12^{-1}24^{1} ,     $  \\
 $ \Delta_{3, 24, \chi_{-24};5}(\tau)  =  1^{1}2^{-1}3^{-5}4^{1}6^{14}12^{-6}24^{2},        $ & $ \Delta_{3, 24, \chi_{-24};6}(\tau)   =  1^{2}2^{-2}3^{-4}4^{4}6^{11}8^{-1}12^{-5}24^{1}.$  \\
 \end{tabular}
\end{center}

 \smallskip
 
The Fourier expansions of the above forms are written as follows:
\begin{equation*}
\Delta_{k,N,\chi}(\tau) ~=~ \sum_{n\ge 1} \tau_{k,N,\chi}(n) q^n, \qquad  \Delta_{k,N,\chi;j}(\tau) ~=~ \sum_{n\ge 1} \tau_{k,N,\chi;j}(n) q^n.
\end{equation*}

\smallskip

\subsection{Sample formulas (4 variables)}

\begin{equation*}
\begin{split}
\delta_4(1^{2} 3^{2}        ;n-1)                   & =                  - \sigma(n) + \sigma(n/2) + 3 \sigma(n/3) -3 \sigma(n/6),  \\ 
\delta_4(2^{4}              ;n-1)                   & =                  -  \sigma(n)  +   3  \sigma(n/2)  -2  \sigma(n/4),      \\ 
\delta_4( 1^{2} 2^{1} 4^{1} ;n-1)                   & =  \sigma_{1; \chi_{8},{\bf 1}}(n),  \\
{\mathcal N}_{s,t}(1^{1} 2^{1} ; 4^{2}       ;n-1)  & =  \sigma_{1; \chi_{8},{\bf 1}}(n),  \\  
{\mathcal N}_{s,t}( 1^{1}      ;2^{2}4^{1}   ;n-1)  & =  \sigma_{1; \chi_{8},{\bf 1}}(n),  \\  
{\mathcal N}_{s,t}(1^{2}       ; 4^{2}       ;n-1 ) & = -  \sigma(n)  -  \sigma(n/2)  +   10  \sigma(n/4)  -8  \sigma(n/8),    \\   
{\mathcal N}_{s,t}(2^{2}       ; 4^{2}       ;n-1)  & = -  \sigma(n)  +   3  \sigma(n/2)  -2  \sigma(n/4),      \\   
{\mathcal N}_{s,t}(2^{1}       ; 2^{2}4^{1}  ;n-1)  & = -  \sigma(n)  +   3  \sigma(n/2)  -2  \sigma(n/4),     \\   
{\mathcal N}_{l,t}(1^{1} ; 2^{1} 6^{1} ;n-1)        & = -  \sigma(n)    -   3  \sigma(n/2)     +   3  \sigma(n/3)    +   4  \sigma(n/4)     +   9  \sigma(n/6)    -   12  \sigma(n/12),    \\   
{\mathcal N}_{l,t}(2^{1} ; 2^{1} 6^{1} ;n-1)        & = -  \sigma(n)    +   3  \sigma(n/2)     -   3  \sigma(n/3)    -   2  \sigma(n/4)    +   9  \sigma(n/6)    -   6  \sigma(n/12).    \\   
\end{split}
\end{equation*}

\subsection{Sample fromulas (6 variables)}

\begin{equation*}
\begin{split}
\delta_6(1^4 2^2               ;n-1)  &   =  \sigma_{2;\chi_{-4},1}(n), \\
\delta_6(1^{5} 3^{1}           ;n-1)  &  =  -\frac{1}{8}\sigma_{2;{\bf 1},\chi_{-3}}(n) + \frac{1}{8}\sigma_{2;{\bf 1},\chi_{-3}}(n/2) + \frac{9}{8}\sigma_{2;\chi_{-3},{\bf 1}}(n) + \frac{9}{8}\sigma_{2;\chi_{-3},{\bf 1}}(n/2),   \\
\delta_6(1^{1} 3^{5}           ;n-1)  &  =  -\frac{1}{8}\sigma_{2;{\bf 1},\chi_{-3}}(n) + \frac{1}{8}\sigma_{2;{\bf 1},\chi_{-3}}(n/2) + \frac{1}{8}\sigma_{2;\chi_{-3},{\bf 1}}(n) + \frac{1}{8}\sigma_{2;\chi_{-3},{\bf 1}}(n/2),   \\
\delta_6(2^4 4^2               ;n-2)  &  =   \sigma_{2;\chi_{-4},{\bf 1}}(n/2), \\
\delta_6(4^6                   ;n-3)  &  =  -\frac{1}{16}\sigma_{2;{\bf 1},\chi_{-4}}(n)  + \frac{1}{16}   \sigma_{2;{\bf 1},\chi_{-4}}(n/2)  + \frac{1}{16}    \sigma_{2;\chi_{-4},{\bf 1}}(n)  -\frac{1}{4}   \sigma_{2;\chi_{-4},{\bf 1}}(n/2),\\
\delta_6(1^{2} 2^{1} 4^{3}     ;n-2)  &  =   \frac{1}{6} \sigma_{2;\chi_{-8},{\bf 1}}(n) - \frac{1}{6} f_{3,8,\chi_{-8}}(n),  \\
{\mathcal N}_{l,t}(1^{1}      ;1^{2} 3^{2}      ;n-1)     & =  -\frac{1}{2}\sigma_{2;{\bf 1},\chi_{-3}}(n) + \frac{1}{2}\sigma_{2;{\bf 1},\chi_{-3}}(n/2) + \frac{3}{2}\sigma_{2;\chi_{-3},{\bf 1}}(n) + \frac{3}{2}\sigma_{2;\chi_{-3},{\bf 1}}(n/2),    \\
{\mathcal N}_{l,t}(2^{1}      ;1^{2} 3^{2}      ;n-1)     & =  \frac{1}{4}\sigma_{2;{\bf 1},\chi_{-3}}(n)  - \frac{1}{4}\sigma_{2;{\bf 1},\chi_{-3}}(n/2) + \frac{3}{4}\sigma_{2;\chi_{-3},{\bf 1}}(n) + \frac{3}{4}\sigma_{2;\chi_{-3},{\bf 1}}(n/2),    \\ 
\end{split}
\end{equation*}
\begin{equation*}
\begin{split}
{\mathcal N}_{s,t}(1^2         ;2^4                ;n-1)  & =   \sigma_{2;\chi_{-4},1}(n), \\
{\mathcal N}_{s,t}(1^{4}       ;4^2                ;n-1)  & =  -\sigma_{2;{\bf 1},\chi_{-4}}(n)  +   \sigma_{2;{\bf 1},\chi_{-4}}(n/2)  + 2    \sigma_{2;\chi_{-4},{\bf 1}}(n),    \\
{\mathcal N}_{s,t}(2^{4}       ;4^2                ;n-1)  & =  \sigma_{2;{\bf 1},\chi_{-4}}(n)     \sigma_{2;{\bf 1},\chi_{-4}}(n/2)     + \sigma_{2;\chi_{-4},{\bf 1}}(n)   - 4  \sigma_{2;\chi_{-4},{\bf 1}}(n/2), \\  
{\mathcal N}_{s,t}(1^{2} 2^{2} ;4^2                ;n-1)  & =  \sigma_{2;\chi_{-4},{\bf 1}}(n),                                                                                                                  \\      
{\mathcal N}_{s,t}(1^{2}       ;4^4                ;n-2)  & =  -\frac{1}{4}\sigma_{2;{\bf 1},\chi_{-4}}(n)   + \frac{1}{4}  \sigma_{2;{\bf 1},\chi_{-4}}(n/2)    + \frac{1}{4}  \sigma_{2;\chi_{-4},{\bf 1}}(n), \\      
{\mathcal N}_{s,t}(2^{2}       ;4^4                ;n-2)  & =   \sigma_{2;\chi_{-4},{\bf 1}}(n/2),                                                \\                                                                     
{\mathcal N}_{s,t}(2^{2}       ;2^4                ;n-1)  & =   \sigma_{2;\chi_{-4},{\bf 1}}(n)  -4    \sigma_{2;\chi_{-4},{\bf 1}}(n/2),         \\                                                                     
{\mathcal N}_{s,t}(1^{1} 2^{1} ;1^2 2^{1} 4^{1}    ;n-1)  & =   \sigma_{2;\chi_{-4},{\bf 1}}(n),                                                  \\  
{\mathcal N}_{s,t}(1^{1} 2^{3} ;4^{2}              ;n-1)  & =   \frac{2}{3} \sigma_{2;\chi_{-8},{\bf 1}}(n) + \frac{1}{3} a_{3,8,\chi_{-8}} (n),  \\
{\mathcal N}_{s,t}(1^{3} 2^{1} ;4^{2}              ;n-1)  & =   \frac{4}{3} \sigma_{2;\chi_{-8},{\bf 1}}(n) - \frac{1}{3} a_{3,8,\chi_{-8}} (n),  \\
{\mathcal N}_{s,t}(1^{2}       ;1^{2} 2^{1} 4^{1}  ;n-1)  & =   \frac{4}{3} \sigma_{2;\chi_{-8},{\bf 1}}(n) - \frac{1}{3} a_{3,8,\chi_{-8}} (n),  \\
\end{split}
\end{equation*}
\begin{equation*}
\begin{split}
{\mathcal N}_{s,t}(2^{2}       ;1^{2} 2^{1} 4^{1}  ;n-1)  & =   \frac{2}{3} \sigma_{2;\chi_{-8},{\bf 1}}(n) + \frac{1}{3} a_{3,8,\chi_{-8}} (n),  \\
{\mathcal N}_{s,t}(1^{1} 2^{1} ;2^{4}              ;n-1)  & =   \frac{2}{3} \sigma_{2;\chi_{-8},{\bf 1}}(n) + \frac{1}{3} a_{3,8,\chi_{-8}} (n),  \\
{\mathcal N}_{s,t}(1^{1} 2^{1} ;4^{4}              ;n-2)  & =   \frac{1}{6} \sigma_{2;\chi_{-8},{\bf 1}}(n) - \frac{1}{6} a_{3,8,\chi_{-8}}(n),  \\
{\mathcal N}_{s,t}(1^{2} 2^{1}; 2^{2} 4^{1}    ;n-1)      & = \sigma_{2;\chi_{-4},{\bf 1}}(n),                                       \\
{\mathcal N}_{s,t}( 2^{3}; 2^{2} 4^{1}         ;n-1)      & = \sigma_{2;\chi_{-4},{\bf 1}}(n) - 4 \sigma_{2;\chi_{-4},{\bf 1}}(n/2), \\ 
{\mathcal N}_{s,t}( 1^{3}; 2^{2} 4^{1}         ;n-1)      & = \frac{4}{3} \sigma_{2;\chi_{-8},{\bf 1}}(n) - \frac{1}{3} a_{3,8,\chi_{-8}} (n), \\
{\mathcal N}_{s,t}(1^{1} 2^{2}; 2^{2} 4^{1}    ;n-1)      & = \frac{2}{3} \sigma_{2;\chi_{-8},{\bf 1}}(n) + \frac{1}{3} a_{3,8,\chi_{-8}} (n), \\
{\mathcal N}_{l,s,t}(1^{1} ; 1^{1}3^{1}  ; 2^{1}6^{1} ;n-1)   &  = - \frac{1}{2}  \sigma_{2;{\bf 1 },\chi_{-3}}(n) +   \frac{1}{2}   \sigma_{2;{\bf 1 },\chi_{-3}}(n/2) +       \frac{3}{2}   \sigma_{2;\chi_{-3},{\bf 1 }}(n)  +   \frac{3}{2}   \sigma_{2;\chi_{-3},{\bf 1 }}(n/2),     \\ 
{\mathcal N}_{l,s,t}(1^{1} ; 1^{2}       ; 2^{1}6^{1} ;n-1)   &  =    \frac{16}{7}   \sigma_{2;\chi_{-4},{\bf 1 }}(n) +  \frac{72}{7}    \sigma_{2;\chi_{-4},{\bf 1 }}(n/3) - \frac{9}{7}   \tau_{3,12,\chi_{-4};1} (n) - \frac{30}{7}  \tau_{3,12,\chi_{-4};2} (n),  \\ 
{\mathcal N}_{l,s,t}(1^{1} ; 1^{2}       ; 4^{2}      ;n-1)   &  =  \frac{1}{4}    \sigma_{2;{\bf 1 },\chi_{-3}}(n) - \frac{1}{2}    \sigma_{2;{\bf 1 },\chi_{-3}}(n/2) - \frac{7}{4}   \sigma_{2;{\bf 1 },\chi_{-3}}(n/4) +  2   \sigma_{2;{\bf 1 },\chi_{-3}}(n/8)              \\ & \quad   +  \frac{9}{4}      \sigma_{2;\chi_{-3},{\bf 1 }}(n)   +  \frac{9}{2}      \sigma_{2;\chi_{-3},{\bf 1 }}(n/2) - \frac{117}{4}   \sigma_{2;\chi_{-3},{\bf 1 }}(n/4) + 90    \sigma_{2;\chi_{-3},{\bf 1 }}(n/8)      \\ & \quad   -3             \tau_{3,12,\chi_{-3}} (n)  + \frac{9}{2}    \tau_{3,12,\chi_{-3}} (n/2) +  \frac{3}{2}   \tau_{3,24,\chi_{-3};1} (n), \\ 
{\mathcal N}_{l,s,t}(8^{1} ; 1^{2}       ; 2^{1}6^{1} ;n-1)   &  =  \frac{5}{14}  \sigma_{2;\chi_{-4},{\bf 1 }}(n)  - \frac{9}{14}  \sigma_{2;\chi_{-4},{\bf 1 }}(n/3) +  \frac{9}{14}   \tau_{3,12,\chi_{-4};1} (n) - \frac{6}{7}   \tau_{3,12,\chi_{-4};2} (n)       \\ & \quad   +  6             \tau_{3,12,\chi_{-4};3} (n)      +  18            \tau_{3,12,\chi_{-4};4} (n),   \\                                                                                                                  
{\mathcal N}_{l,s,t}(1^{1} ; 1^{1}6^{1}  ; 4^{2}      ;n-1)   &  =    \frac{32}{39}   \sigma_{2;\chi_{-8},{\bf 1 }}(n) +    \frac{24}{13}    \sigma_{2;\chi_{-8},{\bf 1 }}(n/3) - \frac{47}{39}   \tau_{3,8,\chi_{-8}} (n)      - \frac{6}{13}    \tau_{3,8,\chi_{-8}} (n/3) \\ & \quad     +    \frac{14}{13}    \tau_{3,24,\chi_{-8};1} (n) +    \frac{21}{26}   \tau_{3,24,\chi_{-8};2} (n) +    \frac{33}{26}    \tau_{3,24,\chi_{-8};3} (n) - \frac{23}{13}  \tau_{3,24,\chi_{-8};4} (n),     \\  
{\mathcal N}_{l,s,t}(1^{1} ; 1^{1}2^{1}  ; 4^{2}      ;n-1)   &  =   \frac{36}{23}    \sigma_{2;\chi_{-24},{\bf 1 }}(n) - \frac{4}{23}   \sigma_{2;\chi_{-3},\chi_{8}}(n) - \frac{51}{92}    \tau_{3,24,\chi_{-24};1} (n) +     \frac{3}{46}      \tau_{3,24,\chi_{-24};2} (n) \\ & \quad      +     \frac{51}{23}    \tau_{3,24,\chi_{-24};3} (n) +     \frac{15}{23}    \tau_{3,24,\chi_{-24};4} (n) - \frac{153}{92}    \tau_{3,24,\chi_{-24};5} (n) - \frac{51}{46}   \tau_{3,24,\chi_{-24};6} (n).  \\  
\end{split}
\end{equation*}

\section{Appendix}

As mentioned in \S 1.1.1, to get explicit formulas given by \eqref{formula1} for the list of forms in Tables 2(a) and 2(b), we use 
the following tables. These tables (Tables 3 to 14) give the required coefficients `$t_i$' that appear in the formulas \eqref{formula1}.

\smallskip

\subsection{Tabular format (4 variables)}

We provide the sample formula for representation number in tables as follows.

\begin{center}

\renewcommand{\arraystretch}{1.2}
\begin{tabularx}{\textwidth}{|p{3.6cm}| p{0.8cm}|p{0.8cm}|p{0.8cm}|p{0.8cm}|p{0.8cm}|p{0.8cm}|p{0.8cm}|p{0.8cm}|p{0.8cm}|p{0.8cm}|p{0.8cm}|p{0.8cm}|}
\caption{Coefficient of linear span}\\ 
\hline
     
                            \multicolumn{1}{|c|}{\textbf{Representation number}}  & \multicolumn{6}{|c|}{\textbf{Coefficient of linear span }}\\ \cline{2-7} 
    {} & {$t_1$} & {$t_2$} & {$t_3$} & {$t_4$} & {$t_5$} & {$t_6$} \\ \hline 
     
\endfirsthead
     \caption{Coefficient of linear span. (continued\ldots)}\\
      \hline
                              \multicolumn{1}{|c|}{\textbf{Representation number}}  & \multicolumn{6}{|c|}{\textbf{Coefficient of linear span }}\\ \cline{2-7} 
    {} & {$t_1$} & {$t_2$} & {$t_3$} & {$t_4$} & {$t_5$} & {$t_6$}  \\ \hline 
     
\endhead
      \hline
      \multicolumn{6}{|r|}{{Continued on Next Page\ldots}}     \\ \hline
      
\endfoot
    
\endlastfoot
\hline
$ {\mathcal N}_{s,t}(1^{1}3^{1}; 2^{1}6^{1}; n-1) $ & $  -\frac{1}{24} $  & $  \frac{1}{24} $  & $  \frac{1}{8} $  & $  0 $  & $  -\frac{1}{8} $  & $  0 $ \\ \hline
$ {\mathcal N}_{s,t}(1^{1}     ; 2^{1}3^{2}; n-1) $ & $  -\frac{1}{24} $  & $  \frac{1}{24} $  & $  \frac{1}{8} $  & $  0 $  & $  -\frac{1}{8} $  & $  0 $ \\ \hline
$ {\mathcal N}_{s,t}(3^{1}     ; 1^{2}6^{1}; n-1) $ & $  -\frac{1}{24} $  & $  \frac{1}{24} $  & $  \frac{1}{8} $  & $  0 $  & $  -\frac{1}{8} $  & $  0 $ \\ \hline
$ {\mathcal N}_{l,t}(1^{1}     ; 2^{1}6^{1}; n-1) $ & $  -\frac{1}{24} $  & $  -\frac{1}{8} $  & $  \frac{1}{8} $  & $  \frac{1}{6} $  & $  \frac{3}{8} $  & $  -\frac{1}{2} $ \\ \hline
$ {\mathcal N}_{l,t}(2^{1}     ; 2^{1}6^{1}; n-1) $ & $  -\frac{1}{24} $  & $  \frac{1}{8} $  & $  -\frac{1}{8} $  & $  -\frac{1}{12} $  & $  \frac{3}{8} $  & $  -\frac{1}{4} $ \\ \hline
$ {\mathcal N}_{l,t}(4^{1}     ; 2^{1}6^{1}; n-1) $ & $  -\frac{1}{24} $  & $  \frac{1}{8} $  & $  \frac{1}{8} $  & $  -\frac{1}{12} $  & $  -\frac{3}{8} $  & $  \frac{1}{4} $ \\
\hline

\end{tabularx}

\end{center}

\begin{center}

\renewcommand{\arraystretch}{1.2}
\begin{tabularx}{\textwidth}{|p{3.6cm}| p{0.8cm}|p{0.8cm}|p{0.8cm}|p{0.8cm}|}
\caption{Coefficient of linear span}\\ 
\hline
     
                            \multicolumn{1}{|c|}{\textbf{Representation number}}  & \multicolumn{4}{|c|}{\textbf{Coefficient of linear span }}\\ \cline{2-5} 
    {} & {$t_1$} & {$t_2$} & {$t_3$} & {$t_4$} \\ \hline 
     
\endfirsthead
     \caption{Coefficient of linear span. (continued\ldots)}\\
      \hline
                              \multicolumn{1}{|c|}{\textbf{Representation number}}  & \multicolumn{4}{|c|}{\textbf{Coefficient of linear span}}\\ \cline{2-5} 
    {} & {$t_1$} & {$t_2$} & {$t_3$} & {$t_4$}  \\ \hline 
     
\endhead
      \hline
      \multicolumn{5}{|r|}{{Continued on Next Page\ldots}}     \\ \hline
      
\endfoot
    
\endlastfoot
\hline
$ \delta_{4}        (1^{1}2^{2}3^{1}   ;n-1)   $  &  $  0 $  &  $  \frac{3}{4} $  &  $  \frac{1}{4} $  &  $  0  $ \\ \hline
$ \delta_{4}        (1^{1}3^{1}6^{2}   ;n-2)   $  &  $  0 $  &  $  \frac{1}{4} $  &  $  -\frac{1}{4} $  &  $  0  $ \\ \hline
$ {\mathcal N}_{s,t}(1^{2}; 2^{1}6^{1} ;n-1)   $  &  $  0 $  &  $  \frac{3}{2} $  &  $  -\frac{1}{2} $  &  $  0  $ \\ \hline
$ {\mathcal N}_{s,t}(3^{2}; 2^{1}6^{1} ;n-1)   $  &  $  0 $  &  $  \frac{1}{2} $  &  $  \frac{1}{2} $  &  $  0  $ \\ \hline
$ {\mathcal N}_{s,t}(3^{1} ; 2^{1}3^{2};n-1)   $  &  $  0 $  &  $  \frac{1}{2} $  &  $  \frac{1}{2} $  &  $  0  $ \\ \hline
$ {\mathcal N}_{s,t}(1^{1} ; 1^{2}6^{1};n-1)   $  &  $  0 $  &  $  \frac{3}{2} $  &  $  -\frac{1}{2} $  &  $  0  $ \\ \hline

\end{tabularx}

\end{center}

\begin{center}
\renewcommand{\arraystretch}{1.2}
\begin{tabularx}{\textwidth}{|p{3.6cm}| p{0.8cm}|p{0.8cm}|p{0.8cm}|p{0.8cm}|p{0.8cm}|p{0.8cm}|p{0.8cm}|p{0.8cm}|p{0.8cm}|}
\caption{Coefficient of linear span}\\ 
\hline
     
                            \multicolumn{1}{|c|}{\textbf{Representation number}}  & \multicolumn{9}{|c|}{\textbf{Coefficient of linear span }}\\ \cline{2-10} 
    {} & {$t_1$} & {$t_2$} & {$t_3$} & {$t_4$} & {$t_5$} & {$t_6$} & {$t_7$} & {$t_8$} & {$t_9$} \\ \hline 
     
\endfirsthead
     \caption{Coefficient of linear span. (continued\ldots)}\\
      \hline
                              \multicolumn{1}{|c|}{\textbf{Representation number}}  & \multicolumn{9}{|c|}{\textbf{Coefficient of linear span }}\\ \cline{2-10} 
    {} & {$t_1$} & {$t_2$} & {$t_3$} & {$t_4$} & {$t_5$} & {$t_6$} & {$t_7$} & {$t_8$} & {$t_9$}  \\ \hline 
     
\endhead
      \hline
      \multicolumn{9}{|r|}{{Continued on Next Page\ldots}}     \\ \hline
      
\endfoot
    
\endlastfoot
\hline

$ {\mathcal N}_{s,t}(3^{2}      ; 4^{2}         ;n-1)      $ &  $  -\frac{1}{48}$  &  $  \frac{1}{16}$  &  $  \frac{1}{16}$  &  $  -\frac{5}{24}$  &  $  -\frac{3}{16}$  &  $  \frac{1}{6}$  &  $  \frac{5}{8}$  &  $  -\frac{1}{2}$  &  $  \frac{1}{2}$  \\ \hline 
$ {\mathcal N}_{s,t}(6^{2}      ; 4^{2}         ;n-1)      $ &  $  -\frac{1}{48}$  &  $  \frac{1}{16}$  &  $  \frac{1}{16}$  &  $  -\frac{1}{24}$  &  $  -\frac{3}{16}$  &  $  0$  &  $  \frac{1}{8}$  &  $  0$  &  $  \frac{1}{2}$  \\ \hline 
$ {\mathcal N}_{s,t}(2^{1}6^{1} ; 2^{1}6^{1}    ;n-1)      $ &  $  -\frac{1}{48}$  &  $  \frac{1}{16}$  &  $  -\frac{1}{16}$  &  $  -\frac{1}{24}$  &  $  \frac{3}{16}$  &  $  0$  &  $  -\frac{1}{8}$  &  $  0$  &  $  \frac{1}{2}$  \\ \hline 
$ {\mathcal N}_{s,t}(2^{1}      ;4^{1}6^{2}     ;n-2)      $ &  $  0$  &  $  -\frac{1}{24}$  &  $  0$  &  $  \frac{1}{24}$  &  $  \frac{1}{8}$  &  $  0$  &  $  -\frac{1}{8}$  &  $  0$  &  $  0$  \\ \hline 
$ {\mathcal N}_{s,t}(6^{1}      ;1^{1}3^{1}4^{1};n-1)      $ &  $  -\frac{1}{48}$  &  $  \frac{1}{48}$  &  $  \frac{1}{16}$  &  $  0$  &  $  -\frac{1}{16}$  &  $  0$  &  $  0$  &  $  0$  &  $  \frac{1}{2}$  \\ \hline 
$ {\mathcal N}_{l,t}(8^{1}      ; 2^{1}6^{1}    ;n-1)      $ &  $  -\frac{1}{96}$  &  $  \frac{1}{32}$  &  $  -\frac{1}{32}$  &  $  -\frac{1}{48}$  &  $  \frac{3}{32}$  &  $  0$  &  $  -\frac{1}{16}$  &  $  0$  &  $  \frac{3}{4}$  \\
\hline 

\end{tabularx}

\end{center}
\begin{center}

\renewcommand{\arraystretch}{1.2}
\begin{tabularx}{\textwidth}{|p{3.6cm}| p{0.8cm}|p{0.8cm}|p{0.8cm}|p{0.8cm}|p{0.8cm}|p{0.8cm}|}
\caption{Coefficient of linear span}\\ 
\hline
     
                            \multicolumn{1}{|c|}{\textbf{Representation number}}  & \multicolumn{6}{|c|}{\textbf{Coefficient of linear span }}\\ \cline{2-7} 
    {} & {$t_1$} & {$t_2$} & {$t_3$} & {$t_4$} & {$t_5$} & {$t_6$} \\ \hline 
     
\endfirsthead
     \caption{Coefficient of linear span. (continued\ldots)}\\
      \hline
                              \multicolumn{1}{|c|}{\textbf{Representation number}}  & \multicolumn{6}{|c|}{\textbf{Coefficient of linear span }}\\ \cline{2-7} 
    {} & {$t_1$} & {$t_2$} & {$t_3$} & {$t_4$} & {$t_5$} & {$t_6$}  \\ \hline 
     
\endhead
      \hline
      \multicolumn{6}{|r|}{{Continued on Next Page\ldots}}     \\ \hline
      
\endfoot
    
\endlastfoot
\hline
$ {\mathcal N}_{s,t}(3^{1}6^{1} ; 4^{2}     ;n-1)  $ & $  0 $ & $  0 $ & $  \frac{2}{5} $ & $  -\frac{3}{5} $ & $  \frac{2}{5} $ & $  \frac{1}{5} $ \\ \hline
$ {\mathcal N}_{s,t}(1^{1}6^{1} ; 2^{1}6^{1};n-1)  $ & $  0 $ & $  0 $ & $  \frac{4}{5} $ & $  -\frac{6}{5} $ & $  \frac{4}{5} $ & $  -\frac{3}{5} $ \\ \hline
$ {\mathcal N}_{s,t}(2^{1}3^{1} ; 2^{1}6^{1};n-1)  $ & $  0 $ & $  0 $ & $  \frac{2}{5} $ & $  \frac{12}{5} $ & $  \frac{2}{5} $ & $  \frac{1}{5} $ \\ \hline
$ {\mathcal N}_{s,t}(1^{1} ; 4^{1}6^{2}     ;n-2)  $ & $  0 $ & $  0 $ & $  \frac{1}{5} $ & $  \frac{6}{5} $ & $  \frac{1}{5} $ & $  -\frac{2}{5} $ \\ \hline
$ {\mathcal N}_{s,t}(3^{1} ; 1^{1}3^{1}4^{1};n-1)  $ & $  0 $ & $  0 $ & $  \frac{4}{5} $ & $  -\frac{6}{5} $ & $  -\frac{1}{5} $ & $  \frac{2}{5} $ \\ \hline
$ {\mathcal N}_{s,t}(2^{1} ; 2^{1}3^{2}     ;n-1)  $ & $  0 $ & $  0 $ & $  \frac{2}{5} $ & $  \frac{12}{5} $ & $  \frac{2}{5} $ & $  \frac{1}{5} $ \\ \hline
$ {\mathcal N}_{s,t}(6^{1} ; 1^{2}6^{1}     ;n-1)  $ & $  0 $ & $  0 $ & $  \frac{4}{5} $ & $  -\frac{6}{5} $ & $  \frac{4}{5} $ & $  -\frac{3}{5} $ \\
\hline

\end{tabularx}

\end{center}
\begin{center}

\renewcommand{\arraystretch}{1.2}
\begin{tabularx}{\textwidth}{|p{3.6cm}| p{0.8cm}|p{0.8cm}|p{0.8cm}|p{0.8cm}|p{0.8cm}|p{0.8cm}|p{0.8cm}|p{0.8cm}|}
\caption{Coefficient of linear span}\\ 
\hline
     
                            \multicolumn{1}{|c|}{\textbf{Representation number}}  & \multicolumn{8}{|c|}{\textbf{Coefficient of linear span }}\\ \cline{2-9} 
    {} & {$t_1$} & {$t_2$} & {$t_3$} & {$t_4$} & {$t_5$} & {$t_6$} & {$t_7$} & {$t_8$} \\ \hline 
     
\endfirsthead
     \caption{Coefficient of linear span. (continued\ldots)}\\
      \hline
                              \multicolumn{1}{|c|}{\textbf{Representation number}}  & \multicolumn{8}{|c|}{\textbf{Coefficient of linear span }}\\ \cline{2-9} 
    {} & {$t_1$} & {$t_2$} & {$t_3$} & {$t_4$} & {$t_5$} & {$t_6$} & {$t_7$} & {$t_8$}  \\ \hline 
     
\endhead
      \hline
      \multicolumn{9}{|r|}{{Continued on Next Page\ldots}}     \\ \hline
      
\endfoot
    
\endlastfoot
\hline

$ {\mathcal N}_{s,t}(1^{1}3^{1} ; 4^{2}     ; n-1)        $ & $  -\frac{1}{4} $ & $   \frac{1}{4} $ & $   \frac{3}{4} $ & $   0 $ & $   -\frac{1}{4} $ & $   0 $ & $   \frac{3}{4} $ & $   \frac{3}{4} $ \\  \hline 
$ {\mathcal N}_{s,t}(2^{1}6^{1} ; 4^{2}     ; n-1)        $ & $  0 $ & $   0 $ & $   \frac{3}{4} $ & $   -\frac{3}{2} $ & $   \frac{1}{4} $ & $   \frac{1}{2} $ & $   0 $ & $   0 $ \\  \hline 
$ {\mathcal N}_{s,t}(2^{2} ; 2^{1}6^{1}     ; n-1)        $ & $  0 $ & $   0 $ & $   \frac{3}{2} $ & $   -3 $ & $   -\frac{1}{2} $ & $   -1 $ & $   0 $ & $   0 $ \\  \hline 
$ {\mathcal N}_{s,t}(6^{2} ; 2^{1}6^{1}     ; n-1)        $ & $  0 $ & $   0 $ & $   \frac{1}{2} $ & $   -1 $ & $   \frac{1}{2} $ & $   1 $ & $   0 $ & $   0 $ \\  \hline 
$ {\mathcal N}_{s,t}(2^{1} ; 1^{1}3^{1}4^{1}; n-1)        $ & $  0 $ & $   0 $ & $   \frac{3}{4} $ & $   0 $ & $   \frac{1}{4} $ & $   0 $ & $   0 $ & $   0 $ \\  \hline 
$ {\mathcal N}_{s,t}(6^{1} ; 2^{2}4^{1}     ; n-1)        $ & $  0 $ & $   0 $ & $   \frac{3}{4} $ & $   -\frac{3}{2} $ & $   \frac{1}{4} $ & $   \frac{1}{2} $ & $   0 $ & $   0 $ \\  \hline 
$ {\mathcal N}_{s,t}(6^{1} ; 4^{1}6^{2}     ; n-2)        $ & $  0 $ & $   0 $ & $   0 $ & $   \frac{1}{2} $ & $   0 $ & $   \frac{1}{2} $ & $   0 $ & $   0 $ \\  \hline 
$ {\mathcal N}_{l,t}(1^{1} ; 4^{2}          ; n-1)        $ & $  -\frac{1}{4} $ & $   \frac{1}{4} $ & $   \frac{3}{4} $ & $   3 $ & $   -\frac{1}{4} $ & $   1 $ & $   \frac{3}{4} $ & $   \frac{3}{4} $ \\
\hline 

\end{tabularx}

\end{center}
\begin{center}

\renewcommand{\arraystretch}{1.2}
\begin{tabularx}{\textwidth}{|p{3.6cm}| p{0.8cm}|p{0.8cm}|p{0.8cm}|p{0.8cm}|p{0.8cm}|p{0.8cm}|}
\caption{Coefficient of linear span}\\ 
\hline
     
                            \multicolumn{1}{|c|}{\textbf{Representation number}}  & \multicolumn{6}{|c|}{\textbf{Coefficient of linear span }}\\ \cline{2-7} 
    {} & {$t_1$} & {$t_2$} & {$t_3$} & {$t_4$} & {$t_5$} & {$t_6$} \\ \hline 
     
\endfirsthead
     \caption{Coefficient of linear span. (continued\ldots)}\\
      \hline
                              \multicolumn{1}{|c|}{\textbf{Representation number}}  & \multicolumn{6}{|c|}{\textbf{Coefficient of linear span }}\\ \cline{2-7} 
    {} & {$t_1$} & {$t_2$} & {$t_3$} & {$t_4$} & {$t_5$} & {$t_6$}  \\ \hline 
     
\endhead
      \hline
      \multicolumn{7}{|r|}{{Continued on Next Page\ldots}}     \\ \hline
      
\endfoot
    
\endlastfoot
\hline

$ \delta_{4}         (3^{2}4^{1}6^{1}        ; n-2)          $     &   $ 0 $ & $  \frac{1}{6} $ & $  -\frac{1}{6} $ & $  0 $ & $  \frac{1}{3} $ & $  0 $  \\  \hline
$ {\mathcal N}_{s,t} (1^{1}6^{1} ; 4^{2}     ; n-1)       $     &   $ 0 $ & $  \frac{1}{2} $ & $  -\frac{1}{6} $ & $  0 $ & $  2 $ & $  \frac{2}{3} $  \\  \hline
$ {\mathcal N}_{s,t} (2^{1}3^{1} ; 4^{2}     ; n-1)       $     &   $ 0 $ & $  \frac{1}{2} $ & $  \frac{1}{6} $ & $  0 $ & $  0 $ & $  \frac{1}{3} $  \\  \hline
$ {\mathcal N}_{s,t} (1^{1}2^{1} ; 2^{1}6^{1}; n-1)  $     &   $ 0 $ & $  1 $ & $  \frac{1}{3} $ & $  0 $ & $  0 $ & $  -\frac{1}{3} $  \\  \hline
$ {\mathcal N}_{s,t} (3^{1}6^{1} ; 2^{1}6^{1}; n-1)  $     &   $ 0 $ & $  \frac{1}{3} $ & $  \frac{1}{3} $ & $  0 $ & $  \frac{2}{3} $ & $  \frac{1}{3} $  \\  \hline
$ {\mathcal N}_{s,t} (1^{1} ; 1^{1}3^{1}4^{1}; n-1)  $     &   $ 0 $ & $  1 $ & $  -\frac{1}{3} $ & $  0 $ & $  1 $ & $  \frac{1}{3} $  \\  \hline
$ {\mathcal N}_{s,t} (2^{1} ; 1^{2}6^{1}     ; n-1)       $     &   $ 0 $ & $  1 $ & $  \frac{1}{3} $ & $  0 $ & $  0 $ & $  -\frac{1}{3} $  \\  \hline
$ {\mathcal N}_{s,t} (3^{1} ; 4^{1}6^{2}     ; n-2)       $     &   $ 0 $ & $  \frac{1}{6} $ & $  -\frac{1}{6} $ & $  0 $ & $  \frac{1}{3} $ & $  0 $  \\
\hline

\end{tabularx}
\end{center}

\newpage

\subsection{Tabular format (6 variables)}

Here, we present  the formulas for representation number.   


\begin{center}

 \renewcommand{\arraystretch}{1.2}
\begin{tabularx}{\textwidth}{|p{4cm}| p{0.8cm}|p{0.8cm}|p{0.8cm}|p{0.8cm}|p{0.8cm}|p{0.8cm}|p{0.8cm}|}
\caption{Coefficient of linear span}\\ 
\hline
     
                            \multicolumn{1}{|c|}{\textbf{Representation number}}  & \multicolumn{7}{|c|}{\textbf{Coefficient of linear span }}\\ \cline{2-8} 
    {} & {$t_1$} & {$t_2$} & {$t_3$} & {$t_4$} & {$t_5$} & {$t_6$} & {$t_7$}  \\ \hline 
     
\endfirsthead
     \caption{coefficient in linear span. (continued\ldots)}\\
      \hline
                              \multicolumn{1}{|c|}{\textbf{Representation number}}  & \multicolumn{7}{|c|}{\textbf{Coefficient of linear span }}\\ \cline{2-8} 
    {} & {$t_1$} & {$t_2$} & {$t_3$} & {$t_4$} & {$t_5$} & {$t_6$} & {$t_7$}   \\ \hline 
     
\endhead
      \hline
      \multicolumn{8}{|r|}{{Continued on Next Page\ldots}}     \\ \hline
      
\endfoot
    
\endlastfoot
\hline

$\delta_6(1^{2}2^{1}3^{2}6^{1};n-2)                     $  & $\frac{-1}{16}$  &   $\frac{1}{16}$  &   $  0 $  &   $\frac{3}{16}$  &   $\frac{3}{16}$  &   $  0 $  &   $\frac{-1}{8}$\\ \hline
$\delta_6(2^{5}6^{1}          ;n-2)                     $  & $  0 $  &   $\frac{-1}{8}$  &   $\frac{1}{8}$  &   $  0 $  &   $\frac{9}{8}$  &   $\frac{9}{8}$  &   $  0 $\\ \hline
$\delta_6(2^{1}6^{5}          ;n-4)                     $  & $  0 $  &   $\frac{-1}{8}$  &   $\frac{1}{8}$  &   $  0 $  &   $\frac{1}{8}$  &   $\frac{1}{8}$  &   $  0 $\\ \hline
$\delta_6(2^{3}6^{3}          ;n-3)                     $  & $\frac{-1}{32}$  &   $\frac{-3}{32}$  &   $\frac{1}{8}$  &   $\frac{3}{32}$  &   $\frac{-9}{32}$  &   $\frac{-3}{8}$  &   $\frac{-1}{16}$\\ \hline

${\mathcal N}_{l,t} (1^{2}       ; 2^{1}6^{1}; n-1)     $  & $\frac{-5}{4}$  &   $\frac{-3}{4}$  &   2  &   $\frac{15}{4}$  &   $\frac{-9}{4}$  &   -6  &   $\frac{-3}{2}$\\ \hline
${\mathcal N}_{l,t} (1^{1}2^{1}  ; 2^{1}6^{1}; n-1)     $  & $\frac{1}{4}$  &   $\frac{-9}{4}$  &   2  &   $\frac{3}{4}$  &   $\frac{27}{4}$  &   6  &   $  0 $\\ \hline
${\mathcal N}_{l,t} (1^{1}4^{1}  ; 2^{1}6^{1}; n-1)     $  & $\frac{-1}{8}$  &   $\frac{9}{8}$  &   -1  &   $\frac{3}{8}$  &   $\frac{27}{8}$  &   3  &   $\frac{3}{4}$\\ \hline
${\mathcal N}_{l,t} (2^{2}       ; 2^{1}6^{1}; n-1)     $  & $\frac{-1}{2}$  &   $\frac{-3}{2}$  &   2  &   $\frac{3}{2}$  &   $\frac{-9}{2}$  &   -6  &   $  0 $\\
\hline

\hline

\end{tabularx}

\end{center}
\begin{center}

\renewcommand{\arraystretch}{1.2}
\begin{tabularx}{\textwidth}{|p{4cm}| p{0.8cm}|p{0.8cm}|p{0.8cm}|p{0.8cm}|p{0.8cm}|p{0.8cm}|}
\caption{Coefficient of linear span}\\ 
\hline
     
                            \multicolumn{1}{|c|}{\textbf{Representation number}}  & \multicolumn{6}{|c|}{\textbf{Coefficient of linear span }}\\ \cline{2-7} 
    {} & {$t_1$} & {$t_2$} & {$t_3$} & {$t_4$} & {$t_5$} & {$t_6$}   \\ \hline 
     
\endfirsthead
     \caption{Coefficient of linear span (continued\ldots)}\\
      \hline
                              \multicolumn{1}{|c|}{\textbf{Representation number}}  & \multicolumn{6}{|c|}{\textbf{Coefficient of linear span }}\\ \cline{2-7} 
    {} & {$t_1$} & {$t_2$} & {$t_3$} & {$t_4$} & {$t_5$} & {$t_6$}    \\ \hline 
     
\endhead
      \hline
      \multicolumn{7}{|r|}{{Continued on Next Page\ldots}}     \\ \hline
      
\endfoot
    
\endlastfoot
\hline

$\delta_{6} (1^{4}6^{2}          ; n-2)                  $ & $ 0 $  & $ 0 $  & $ \frac{2}{7} $  & $ \frac{9}{7} $  & $ \frac{-2}{7} $  & $ \frac{-9}{7} $  \\ \hline
$\delta_{6} (1^{1}2^{3}3^{1}6^{1}; n-2)                  $ & $ 0 $  & $ 0 $  & $ \frac{5}{28} $  & $ \frac{-9}{28} $  & $ \frac{-5}{28} $  & $ \frac{-3}{7} $  \\ \hline
$\delta_{6} (1^{1}2^{1}3^{1}6^{3}; n-3)                  $ & $ 0 $  & $ 0 $  & $ \frac{1}{28} $  & $ \frac{15}{28} $  & $ \frac{-1}{28} $  & $ \frac{-2}{7} $  \\ \hline
$\delta_{6} (2^{2}3^{4}          ; n-2)                  $ & $ 0 $  & $ 0 $  & $ \frac{1}{7} $  & $ \frac{-6}{7} $  & $ \frac{-1}{7} $  & $ \frac{-1}{7} $  \\ \hline
$\delta_{6} (3^{4}6^{2}          ; n-3)                  $ & $ 0 $  & $ 0 $  & $ 0 $  & 1  & $ 0 $  & $ 0 $  \\ \hline

${\mathcal N}_{l,t} (1^{1} ; 1^{1}2^{2}3^{1}; n-1)       $  & $ 0 $  & $ 0 $  & $ \frac{10}{7} $  & $ \frac{-18}{7} $  & $ \frac{-3}{7} $  & $ \frac{-3}{7} $  \\ \hline
${\mathcal N}_{l,t} (1^{1} ; 1^{1}3^{1}6^{2}; n-2)       $  & $ 0 $  & $ 0 $  & $ \frac{2}{7} $  & $ \frac{30}{7} $  & $ \frac{-2}{7} $  & $ \frac{-9}{7} $  \\ \hline
${\mathcal N}_{l,t} (2^{1} ; 1^{1}2^{2}3^{1}; n-1)       $  & $ 0 $  & $ 0 $  & $ \frac{4}{7} $  & $ \frac{18}{7} $  & $ \frac{3}{7} $  & $ \frac{3}{7} $  \\
\hline

\hline

\end{tabularx}

\end{center}
\begin{center}

\renewcommand{\arraystretch}{1.2}
\begin{tabularx}{\textwidth}{|p{4.2cm}| p{0.7cm}|p{0.7cm}|p{0.7cm}|p{0.7cm}|p{0.7cm}|p{0.7cm}|p{0.7cm}|p{0.7cm}|p{0.7cm}|p{0.7cm}|p{0.8cm}|p{0.7cm}|}
\caption{Coefficient of linear span}\\ 
\hline
     
                            \multicolumn{1}{|c|}{\textbf{Representation number}}  & \multicolumn{12}{|c|}{\textbf{Coefficient of linear span }}\\ \cline{2-13} 
    {} & {$t_1$} & {$t_2$} & {$t_3$} & {$t_4$} & {$t_5$} & {$t_6$} & {$t_7$} & {$t_8$} & {$t_9$} & {$t_{10}$} & {$t_{11}$} & {$t_{12}$}\\ \hline 
     
\endfirsthead
     \caption{Coefficient of linear span. (continued\ldots)}\\
      \hline
                              \multicolumn{1}{|c|}{\textbf{Representation number}}  & \multicolumn{12}{|c|}{\textbf{Coefficient of linear span }}\\ \cline{2-13} 
    {} & {$t_1$} & {$t_2$} & {$t_3$} & {$t_4$} & {$t_5$} & {$t_6$} & {$t_7$} & {$t_8$} & {$t_9$} & {$t_{10}$} & {$t_{11}$} & {$t_{12}$} \\ \hline 
     
\endhead
      \hline
      \multicolumn{13}{|r|}{{Continued on Next Page\ldots}}     \\ \hline
      
\endfoot
    
\endlastfoot
\hline

$\delta_6 (1^{1}3^{1}4^{2}6^{2}; n-3) $ & $ \frac{1}{64} $  & $ \frac{-1}{64} $  & $ 0 $  & $ 0 $  & $ \frac{-3}{64} $  & $ \frac{-3}{64} $  & $ \frac{-3}{4} $  & 6  & $ \frac{-7}{32} $  & $ 0 $  & $ \frac{1}{12} $  & $ \frac{1}{6} $\\ \hline
$\delta_6 (2^{1}4^{4}6^{1}      ; n-3) $ & $ \frac{1}{128} $  & $ \frac{3}{128} $  & $ \frac{-1}{32} $  & $ 0 $  & $ \frac{9}{128} $  & $ \frac{-27}{128} $  & $ \frac{-27}{32} $  & $ \frac{9}{2} $  & $ \frac{-9}{64} $  & $ \frac{3}{16} $  & $ \frac{1}{16} $  & $ 0 $\\ \hline
$\delta_6 (1^{1}2^{2}3^{1}4^{2}; n-2) $ & $ \frac{-1}{64} $  & $ \frac{1}{64} $  & $ 0 $  & $ 0 $  & $ \frac{27}{64} $  & $ \frac{27}{64} $  & $ 0 $  & $ 0 $  & $ \frac{3}{32} $  & $ \frac{3}{4} $  & $ 0 $  & $ \frac{-1}{2} $\\ \hline

${\mathcal N}_{l,t}(1^{1}8^{1}             ; 2^{1}6^{1}; n-1) $ & $ \frac{1}{16} $  & $ \frac{-9}{16} $  & $ \frac{1}{2} $  & $ 0 $  & $ \frac{-15}{16} $  & $ \frac{9}{16} $  & 6  & -36  & $ \frac{3}{8} $  & $ \frac{3}{2} $  & $ \frac{-1}{2} $  & 2\\ \hline
${\mathcal N}_{l,t}(2^{1}8^{1}             ; 2^{1}6^{1}; n-1) $ & $ \frac{-1}{8} $  & $ \frac{-3}{8} $  & $ \frac{1}{2} $  & $ 0 $  & $ \frac{-15}{8} $  & $ \frac{-27}{8} $  & -6  & 36  & $ \frac{-3}{2} $  & $ \frac{-3}{2} $  & $ \frac{1}{2} $  & 4\\ \hline
${\mathcal N}_{l,t}(4^{1}8^{1}             ; 2^{1}6^{1}; n-1) $ & $ \frac{1}{16} $  & $ \frac{3}{16} $  & $ \frac{-1}{4} $  & $ 0 $  & $ \frac{-15}{16} $  & $ \frac{-27}{16} $  & $ \frac{15}{4} $  & -36  & $ \frac{3}{8} $  & -3  & $ \frac{-1}{2} $  & 2\\ \hline
${\mathcal N}_{l,t}(8^{1}8^{1}             ; 2^{1}6^{1}; n-1) $ & $ \frac{-1}{32} $  & $ \frac{-3}{32} $  & $ \frac{1}{8} $  & $ 0 $  & $ \frac{-33}{32} $  & $ \frac{-45}{32} $  & $ \frac{-21}{8} $  & 18  & $ \frac{-3}{16} $  & $ \frac{-3}{4} $  & $ \frac{1}{4} $  & 2\\
\hline

\hline

\end{tabularx}

\end{center}
\begin{center}

\renewcommand{\arraystretch}{1.2}
\begin{tabularx}{\textwidth}{|p{4.2cm}| p{0.7cm}|p{0.7cm}|p{0.7cm}|p{0.7cm}|p{0.7cm}|p{0.7cm}|p{0.7cm}|p{0.7cm}|p{0.7cm}|p{0.7cm}|p{0.7cm}|p{0.7cm}|}
\caption{Coefficient of linear span}\\ 
\hline
     
                            \multicolumn{1}{|c|}{\textbf{Representation number}}  & \multicolumn{12}{|c|}{\textbf{Coefficient of linear span }}\\ \cline{2-13} 
    {} & {$t_1$} & {$t_2$} & {$t_3$} & {$t_4$} & {$t_5$} & {$t_6$} & {$t_7$} & {$t_8$} & {$t_9$} & {$t_{10}$} & {$t_{11}$} & {$t_{12}$}\\ \hline 
     
\endfirsthead
     \caption{Coefficient of linear span. (continued\ldots)}\\
      \hline
                              \multicolumn{1}{|c|}{\textbf{Representation number}}  & \multicolumn{12}{|c|}{\textbf{Coefficient of linear span }}\\ \cline{2-13} 
    {} & {$t_1$} & {$t_2$} & {$t_3$} & {$t_4$} & {$t_5$} & {$t_6$} & {$t_7$} & {$t_8$} & {$t_9$} & {$t_{10}$} & {$t_{11}$} & {$t_{12}$} \\ \hline 
     
\endhead
      \hline
      \multicolumn{13}{|r|}{{Continued on Next Page\ldots}}     \\ \hline
      
\endfoot
    
\endlastfoot
\hline

$ \delta_6(1^{2}3^{2}4^{2};n-2) $ & $  0 $ & $  0 $ & $  0 $ & $  0 $ &  $ \frac{1}{7} $ & $  0 $ &  $ \frac{9}{14} $ & $  0 $ &  $ \frac{-1}{7} $ &  $ \frac{-9}{14} $ &  $ \frac{1}{2} $ & $  0 $ \\  \hline
$ \delta_6(2^{2}4^{2}6^{2};n-3)  $ & $  0 $ & $  0 $ & $  0 $ & $  0 $ &  $ \frac{1}{14} $ &  $ \frac{-2}{7} $ &  $ \frac{9}{28} $ &  $ \frac{-9}{7} $ &  $ \frac{-1}{14} $ &  $ \frac{-9}{28} $ &  $ \frac{1}{28} $ &  $ \frac{-3}{14} $ \\  \hline
$ \delta_6(4^{2}6^{4}     ;n-4)        $ & $  0 $ & $  0 $ & $  0 $ & $  0 $ & $  0 $ &  $ \frac{1}{7} $ & $  0 $ &  $ \frac{-6}{7} $ & $  0 $ & $  0 $ &  $ \frac{-1}{7} $ &  $ \frac{-1}{7} $ \\  \hline

${\mathcal N}_{l,t} (1^{2}       ; 4^{2}          ;n-1 ) $ &  $ \frac{2}{7} $ &  $ \frac{-2}{7} $ &  $ \frac{-9}{7} $ &  $ \frac{9}{7} $ &  $ \frac{20}{7} $ &  $ \frac{-32}{7} $ &  $ \frac{90}{7} $ &  $ \frac{-144}{7} $ &  $ \frac{-15}{7} $ &  $ \frac{-72}{7} $ &  $ \frac{48}{7} $ & $  0 $ \\  \hline
${\mathcal N}_{l,t} (1^{1}2^{1}  ; 4^{2}          ;n-1 ) $ &  $ \frac{-5}{14} $ &  $ \frac{5}{14} $ &  $ \frac{-9}{14} $ &  $ \frac{9}{14} $ &  $ \frac{5}{7} $ &  $ \frac{40}{7} $ &  $ \frac{-9}{7} $ &  $ \frac{-72}{7} $ &  $ \frac{9}{14} $ &  $ \frac{18}{7} $ &  $ \frac{-18}{7} $ & $  0 $ \\  \hline
${\mathcal N}_{l,t} (1^{1}4^{1}  ; 4^{2}          ;n-1 ) $ &  $ \frac{2}{7} $ &  $ \frac{-2}{7} $ &  $ \frac{-9}{7} $ &  $ \frac{9}{7} $ &  $ \frac{2}{7} $ &  $ \frac{16}{7} $ &  $ \frac{9}{7} $ &  $ \frac{72}{7} $ &  $ \frac{3}{7} $ &  $ \frac{9}{7} $ &  3 &  $ \frac{36}{7} $ \\  \hline
${\mathcal N}_{l,t} (2^{2}       ; 4^{2}          ;n-1 ) $ &  $ \frac{2}{7} $ &  $ \frac{-2}{7} $ &  $ \frac{-9}{7} $ &  $ \frac{9}{7} $ &  $ \frac{8}{7} $ &  $ \frac{-32}{7} $ &  $ \frac{36}{7} $ &  $ \frac{-144}{7} $ &  $ \frac{-3}{7} $ &  $ \frac{-18}{7} $ &  $ \frac{6}{7} $ & $  0 $ \\
\hline

\hline

\end{tabularx}

\end{center}
\begin{center}

\renewcommand{\arraystretch}{1.2}
\begin{tabularx}{\textwidth}{|p{4.2cm}| p{0.8cm}|p{0.8cm}|p{0.8cm}|p{0.8cm}|p{0.8cm}|p{0.8cm}|p{0.8cm}|p{0.8cm}|p{0.8cm}|p{0.8cm}|}
\caption{Coefficient of linear span}\\ 
\hline
     
                            \multicolumn{1}{|c|}{\textbf{Representation number}}  & \multicolumn{10}{|c|}{\textbf{Coefficient of linear span }}\\ \cline{2-11} 
    {} & {$t_1$} & {$t_2$} & {$t_3$} & {$t_4$} & {$t_5$} & {$t_6$} & {$t_7$} & {$t_8$} & {$t_9$} & {$t_{10}$} \\ \hline 
     
\endfirsthead
     \caption{Coefficient of linear span (continued\ldots)}\\
      \hline
                              \multicolumn{1}{|c|}{\textbf{Representation number}}  & \multicolumn{10}{|c|}{\textbf{Coefficient of linear span }}\\ \cline{2-11} 
    {} & {$t_1$} & {$t_2$} & {$t_3$} & {$t_4$} & {$t_5$} & {$t_6$} & {$t_7$} & {$t_8$} & {$t_9$} & {$t_{10}$}  \\ \hline 
     
\endhead
      \hline
      \multicolumn{11}{|r|}{{Continued on Next Page\ldots}}     \\ \hline
      
\endfoot
    
\endlastfoot
\hline

$ \delta_6 (1^{3} 3^{1}4^{1}6^{1} ;n-2) $  &  $  0 $    &  $  0 $    &   $ \frac{8}{39} $   &   $ \frac{6}{13} $   &   $ \frac{-2}{39} $   &   $ \frac{-3}{26} $   &   $ \frac{1}{52} $   &   $ \frac{1}{13} $   &   $ \frac{-3}{52} $   &   $ \frac{-5}{26} $  \\  \hline
$ \delta_6 (2^{3}  3^{2}4^{1}     ;n-2) $  &  $  0 $    &  $  0 $    &   $ \frac{5}{39} $   &   $ \frac{-6}{13} $   &   $ \frac{-5}{156} $   &   $ \frac{-111}{52} $   &   $ \frac{3}{13} $   &   $ \frac{5}{104} $   &   $ \frac{-33}{104} $   &   $ \frac{-3}{52} $  \\  \hline
$ \delta_6 (2^{1}  3^{2}4^{1}6^{2};n-3) $  &  $  0 $    &  $  0 $    &   $ \frac{1}{39} $   &   $ \frac{4}{13} $   &   $ \frac{-1}{156} $   &   $ \frac{35}{52} $   &   $ \frac{-11}{156} $   &   $ \frac{1}{104} $   &   $ \frac{9}{104} $   &   $ \frac{-7}{156} $  \\  \hline

$ {\mathcal N}_{l,t} (1^{1} ; 3^{2}4^{1}6^{1};n-2) $   &  $  0 $    &  $  0 $    &   $ \frac{8}{39} $   &   $ \frac{32}{13} $   &   $ \frac{-2}{39} $   &   $ \frac{31}{13} $   &   $ \frac{-3}{13} $   &   $ \frac{1}{13} $   &   $ \frac{9}{13} $   &   $ \frac{-9}{13} $  \\  \hline
$ {\mathcal N}_{l,t} (2^{1} ; 3^{2}4^{1}6^{1};n-2) $   &  $  0 $    &  $  0 $    &   $ \frac{8}{39} $   &   $ \frac{-20}{13} $   &   $ \frac{-2}{39} $   &   $ \frac{-34}{13} $   &   $ \frac{7}{26} $   &   $ \frac{1}{13} $   &   $ \frac{-21}{26} $   &   $ \frac{4}{13} $  \\  \hline
$ {\mathcal N}_{l,t} (4^{1} ; 3^{2}4^{1}6^{1};n-2) $   &  $  0 $    &  $  0 $    &   $ \frac{2}{39} $   &   $ \frac{8}{13} $   &   $ \frac{-1}{78} $   &   $ \frac{-43}{26} $   &   $ \frac{5}{26} $   &   $ \frac{1}{52} $   &   $ \frac{9}{52} $   &   $ \frac{-11}{26} $  \\  \hline
$ {\mathcal N}_{l,t} (8^{1} ; 3^{2}4^{1}6^{1};n-2) $   &  $  0 $    &  $  0 $    &   $ \frac{2}{39} $   &   $ \frac{-5}{13} $   &   $ \frac{-1}{78} $   &   $ \frac{-17}{26} $   &   $ \frac{5}{26} $   &   $ \frac{1}{52} $   &   $ \frac{9}{52} $   &   $ \frac{-11}{26} $  \\
\hline

\hline

\end{tabularx}

\end{center}
\begin{center}

\renewcommand{\arraystretch}{1.2}
\begin{tabularx}{\textwidth}{|p{4.2cm}| p{0.8cm}|p{0.8cm}|p{0.8cm}|p{0.8cm}|p{0.8cm}|p{0.8cm}|p{0.8cm}|p{0.8cm}|p{0.8cm}|p{0.8cm}|}
\caption{Coefficient of linear span}\\ 
\hline
     
                            \multicolumn{1}{|c|}{\textbf{Representation number}}  & \multicolumn{10}{|c|}{\textbf{Coefficient of linear span }}\\ \cline{2-11} 
    {} & {$t_1$} & {$t_2$} & {$t_3$} & {$t_4$} & {$t_5$} & {$t_6$} & {$t_7$} & {$t_8$} & {$t_9$} & {$t_{10}$} \\ \hline 
     
\endfirsthead
     \caption{Coefficient of linear span (continued\ldots)}\\
      \hline
                              \multicolumn{1}{|c|}{\textbf{Representation number}}  & \multicolumn{10}{|c|}{\textbf{Coefficient of linear span }}\\ \cline{2-11} 
    {} & {$t_1$} & {$t_2$} & {$t_3$} & {$t_4$} & {$t_5$} & {$t_6$} & {$t_7$} & {$t_8$} & {$t_9$} & {$t_{10}$}  \\ \hline 
     
\endhead
      \hline
      \multicolumn{11}{|r|}{{Continued on Next Page\ldots}}     \\ \hline
      
\endfoot
    
\endlastfoot
\hline

$ \delta_6 (1^{2}4^{1}6^{3}          ;n-3) $  &  $  0 $    &   $ \frac{3}{46} $   &   $ \frac{1}{46} $   &  $  0 $    &   $ \frac{-17}{92} $   &   $ \frac{-19}{184} $   &   $ \frac{31}{46} $   &   $ \frac{21}{92} $   &   $ \frac{-39}{92} $   &   $ \frac{-39}{184} $  \\  \hline
$ \delta_6 (1^{1}2^{1}3^{3}4^{1}     ;n-2) $  &  $  0 $    &   $ \frac{3}{23} $   &   $ \frac{-1}{23} $   &  $  0 $    &   $ \frac{1}{46} $   &   $ \frac{9}{184} $   &   $ \frac{-1}{46} $   &   $ \frac{-1}{46} $   &   $ \frac{-3}{46} $   &   $ \frac{-21}{184} $  \\  \hline
$ \delta_6 (1^{2}2^{2}4^{1}6^{1}     ;n-2) $  &  $  0 $    &   $ \frac{9}{46} $   &   $ \frac{-1}{46} $   &  $  0 $    &   $ \frac{-15}{92} $   &   $ \frac{-5}{92} $   &   $ \frac{15}{23} $   &   $ \frac{19}{92} $   &   $ \frac{-45}{92} $   &   $ \frac{-15}{46} $  \\  \hline
$ \delta_6 (3^{2}4^{3}6^{1}          ;n-3) $  &  $  0 $    &   $ \frac{3}{92} $   &   $ \frac{1}{92} $   &  $  0 $    &   $ \frac{29}{184} $   &   $ \frac{1}{92} $   &   $ \frac{-15}{92} $   &   $ \frac{-1}{92} $   &   $ \frac{-39}{184} $   &   $ \frac{-1}{23} $  \\  \hline

$ {\mathcal N}_{l,t} (1^{1} ; 1^{2}2^{1}4^{1};n-1) $    &  $  0 $    &   $ \frac{36}{23} $   &   $ \frac{-4}{23} $   &  $  0 $    &   $ \frac{-51}{92} $   &   $ \frac{3}{46} $   &   $ \frac{51}{23} $   &   $ \frac{15}{23} $   &   $ \frac{-153}{92} $   &   $ \frac{-51}{46} $  \\  \hline
$ {\mathcal N}_{l,t} (2^{1} ; 1^{2}2^{1}4^{1};n-1) $    &  $  0 $    &   $ \frac{18}{23} $   &   $ \frac{2}{23} $   &  $  0 $    &   $ \frac{-63}{92} $   &   $ \frac{-21}{92} $   &   $ \frac{54}{23} $   &   $ \frac{18}{23} $   &   $ \frac{-117}{92} $   &   $ \frac{-39}{92} $  \\  \hline
$ {\mathcal N}_{l,t} (4^{1} ; 1^{2}2^{1}4^{1};n-1) $    &  $  0 $    &   $ \frac{9}{23} $   &   $ \frac{-1}{23} $   &  $  0 $    &   $ \frac{39}{92} $   &   $ \frac{9}{23} $   &   $ \frac{-39}{23} $   &   $ \frac{-27}{46} $   &   $ \frac{117}{92} $   &   $ \frac{39}{46} $  \\
\hline

\end{tabularx}

\end{center}

\end{document}